\documentclass[11pt,letterpaper]{article}
\usepackage{amsthm}
\usepackage{newtxtext,newtxmath}
\usepackage{amsmath, amssymb, bm}
\usepackage[dvipdfmx]{graphicx}
\usepackage{tabularx}
\usepackage[noend]{algpseudocode}
\usepackage{color}

\usepackage{geometry}
\usepackage{algorithm}

\theoremstyle{plain}
\newtheorem{theorem}{Theorem}[section]
\newtheorem{lemma}[theorem]{Lemma}
\theoremstyle{definition}
\newtheorem{definition}[theorem]{Definition}

\newcommand{\rmp}[1]{{#1}_{\mathrm{p}}}
\newcommand{\rmn}[1]{{#1}_{\mathrm{n}}}
\newcommand{\xo}{X_{\mathrm{n}}}
\newcommand{\xn}{X_{\mathrm{n}}}
\newcommand{\xp}{X_{\mathrm{p}}}
\newcommand{\vo}{V_{\mathrm{n}}}
\newcommand{\vn}{V_{\mathrm{n}}}
\newcommand{\vp}{V_{\mathrm{p}}}
\newcommand{\uo}{U_{\mathrm{n}}}
\newcommand{\un}{U_{\mathrm{n}}}
\newcommand{\up}{U_{\mathrm{p}}}
\newcommand{\ocs}{odd-cycle symmetric}

\newcommand{\sm}{\setminus}
\newcommand{\ktt}{K_{t,t}}
\newcommand{\sd}{\triangle}
\newcommand{\order}[1]{\mathrm{O}( #1 )}

\newcommand{\p}[1]{#1\sp{+}}
\newcommand{\m}[1]{#1\sp{-}}

\newcommand{\h}[1]{\hat{#1}}

\newcommand{\U}{\mathcal{U}}

\newcommand{\ufeas}{$\U$-feasible}

\newcommand{\ZZ}{\mathbf{Z}}
\newcommand{\RR}{\mathbf{R}}

\renewcommand{\c}[1]{$C_{\le {#1}}$}
\newcommand{\matroid}{\mathbf{M}}
\newcommand{\Matroid}{\mathbf{M}}
\newcommand{\Cir}{\mathcal{C}}
\newcommand{\cir}{\mathcal{C}}

\renewcommand{\qedsymbol}{\rule{1ex}{2ex}}

\begin{document}

\title{Excluded $t$-factors in Bipartite Graphs: \\
Unified Framework for Nonbipartite Matchings, \\
Restricted $2$-matchings, and 
Matroids\footnote{A preliminary version of this paper \cite{Tak17ipco} appears in Proceedings of the 19th IPCO.}}

\author{Kenjiro Takazawa\thanks{Department of Industrial and Systems Engineering, Faculty of Science and Engineering, Hosei University, Tokyo 184-8584, Japan.  
{\tt takazawa@hosei.ac.jp}}
}

\date{August, 2017}

\maketitle

\begin{abstract}
We propose a framework for optimal $t$-matchings excluding the prescribed $t$-factors in bipartite graphs. 
The proposed framework is a generalization of the nonbipartite matching problem 
and includes several problems,  
such as the 
triangle-free $2$-matching, 
square-free $2$-matching, 
even factor, 
and arborescence problems. 
In this paper, 
we demonstrate a unified understanding of these problems by 
commonly extending previous important results.  
We 
solve 
our problem under a reasonable assumption, 
which is sufficiently broad to include the specific problems listed above. 
We first present a min-max theorem and a combinatorial algorithm for the unweighted version. 
We then provide a linear programming formulation with dual integrality 
and a primal-dual algorithm for the weighted version. 
A key ingredient of the proposed algorithm is a technique to shrink forbidden structures, 
which corresponds to 
the techniques of 
shrinking odd cycles, triangles, squares, and directed cycles 
in Edmonds' blossom algorithm, 
a triangle-free $2$-matching algorithm, 
a square-free $2$-matching algorithm, 
and 
an arborescence algorithm, 
respectively. 
\end{abstract}

\section{Introduction}

Since matching theory \cite{LP09} was established, 
a number of generalizations of the matching problem have been proposed, 
including 
path-matchings \cite{CG97}, 
even factors \cite{CG01,Pap07,PS04}, 
triangle-free $2$-matchings \cite{CP80,Pap04}, 
square-free $2$-matchings \cite{Hart06,Pap07}, 
$\ktt$-free $t$-matchings \cite{Fra03}, 
$K_{t+1}$-free $t$-matchings \cite{BV10}, 
2-matchings covering prescribed edge cuts \cite{BIT13,KS08}, 
and 
\ufeas{} $2$-matchings \cite{Tak17}. 
For most of these generalizations,  
important results in matching theory can be extended, 
such as a min-max theorem, polynomial algorithms, and a linear programming formulation with dual integrality. 
However, 
while some similar structures are found, 
in most cases, 
they have been studied separately and 
few connections among those similar structures have been identified. 

In this paper, 
we propose a new framework of \emph{optimal $t$-matchings excluding prescribed $t$-factors}, 
to demonstrate a unified understanding of these generalizations. 
The proposed framework includes 
all of the above generalizations  
and 
the arborescence problem. 
Furthermore, 
it includes 
the traveling salesman problem (TSP)\@. 
This broad coverage implies some intractability of the framework; 
however we propose a tractable class that 
includes most of the efficiently solvable classes of the above problems. 

Our main contributions are a min-max theorem and 
a combinatorial polynomial algorithm that commonly extend those for 
the matching and triangle-free $2$-matching problem in nonbipartite graphs, 
the square-free $2$-matching problem in bipartite graphs, 
and 
the arborescence problem in directed graphs. 
A key ingredient of the proposed algorithm is a technique to shrink the excluded $t$-factors. 
This technique commonly extends the techniques used to 
shrink odd cycles,  
triangles, 
squares, 
and directed cycles 
in a matching algorithm \cite{Edm65}, 
a triangle-free $2$-matching algorithm \cite{CP80}, 
a square-free $2$-matching algorithms in bipartite graphs \cite{Hart06,Pap07}, 
and 
an arborescence algorithm \cite{CL65,Edm67}. 
respectively. 
We demonstrate that the proposed framework is 
tractable in the class where this shrinking technique works.

\subsection{Previous Work}

The problems most relevant to our work are the \emph{even factor}, 
\emph{triangle-free $2$-matching}, 
and 
\emph{square-free $2$-matching problems}. 

\subsubsection{Even factor}

The even factor problem \cite{CG01} is a generalization of the nonbipartite matching problem, 
which admits a further generalization: 
the basic/independent even factor problem \cite{CG01,IT08} is 
a common generalization with matroid intersection. 
The origin of the even factor problem is the \emph{independent path-matching problem} \cite{CG97}, 
which is a common generalization of the nonbipartite matching and matroid intersection problems. 
In \cite{CG97}, 
a min-max theorem, 
totally dual integral polyhedral description, 
and 
polynomial solvability by the ellipsoid method were presented. 
These were followed by further analysis of the min-max relation \cite{FS02} 
and Edmonds-Gallai decomposition \cite{SS04}.
A combinatorial approach to the path-matchings was proposed in \cite{SW02} 
and completed by Pap \cite{Pap07}, 
who addressed a further generalization, 
\emph{the even factor problem} \cite{CG01}.

Here, let $D=(V,A)$ be a digraph. 
A subset of arcs $F \subseteq A$ is called a \emph{path-cycle factor} if 
it is a vertex-disjoint collection of directed cycles (dicycles) and directed paths (dipaths). 
Equivalently, 
an arc subset $F$ is a path-cycle factor 
if, 
in the subgraph $(V,F)$, 
the indegree and outdegree of every vertex are at most one. 
An \emph{even factor} is a path-cycle factor excluding 
dicycles of odd length (odd dicycles). 

While 
the maximum even factor problem is NP-hard, 
in \emph{\ocs{}} digraphs 
it enjoys 
min-max theorems \cite{CG01,PS04}, 
the Edmonds-Gallai decomposition \cite{PS04}, 
and 
polynomial-time algorithms \cite{CG01,Pap07}. 
A digraph is called \emph{\ocs{}} if every odd dicycle has its reverse dicycle. 
Moreover, 
a maximum-weight even factor can be found in polynomial time 
in \ocs{} weighted digraphs, 
which are \ocs{} digraphs with arc-weight such that 
the total weight of the arcs in an odd dicycle is 
equal to that of its reverse dicycle. 
The maximum-weight matching problem is straightforwardly reduced to 
the maximum-weight even factor problem in \ocs{} weighted digraphs (see Sect.\ \ref{SECbief}). 
The assumption of odd-cycle symmetry of (weighted) digraphs is supported by 
its relation to 
discrete convexity \cite{KT09}. 

The independent even factor problem is 
a common generalization of the 
even factor and matroid intersection problems. 
In \ocs{} digraphs 
it 
admits 
combinatorial polynomial algorithms \cite{CG01,IT08} and 
a decomposition theorem \cite{IT08}, 
which extends the Edmonds-Gallai decomposition and 
the principal partition for matroid intersection \cite{Iri79,IF81}. 
In \ocs{} weighted digraphs, 
a linear program with dual integrality and 
a combinatorial algorithm for the weighted independent even factor problem 
have been presented in \cite{Tak12wief}. 

The results are summarized in Table \ref{TABef}. 
For more details, 
readers are referred to a survey paper \cite{Tak10}. 

	\begin{table}
	\caption{Results for path-matchings and even factors.  
	(E), (A), (C) denote the ellipsoid method, an algebraic algorithm, and a combinatorial algorithm, 
	respectively.}

	\label{TABef}

	\centering
	\begin{tabular}{|l|l|l|}
	\hline
	 								& Path-matchings 						& Independent path-matchings 				\\ \hline
	Min-max theorem 				& Cunningham--Geelen \cite{CG97}  		& Cunningham--Geelen \cite{CG97} 		\\
									& Frank--Szeg\H{o} \cite{FS02}	 		& 										\\
	Algorithm 						& Cunningham--Geelen \cite{CG97} (E) 	& Cunningham--Geelen \cite{CG97} (E) 	\\
	Decomposition theorem 			& Spille--Szeg\H{o} \cite{SS04} 		& Iwata--Takazawa \cite{IT08} 		\\
	LP formulation					& Cunningham--Geelen \cite{CG97} 		& Cunningham--Geelen \cite{CG97} 		\\
	Algorithm (Weighted) 			& Cunningham--Geelen \cite{CG97} (E) 	& Cunningham--Geelen \cite{CG97} (E) 	\\\hline\hline
	 								& Even factors							& Independent even factors \\ \hline
	Min-max theorem 				& Cunningham--Geelen \cite{CG01}	 	& Iwata--Takazawa \cite{IT08} \\
									& Pap--Szeg\H{o} \cite{PS04} 			& \\
	Algorithm 						& Cunningham--Geelen \cite{CG01} (A)  	& Iwata--Takazawa \cite{IT08} (C) \\
									& Pap \cite{Pap07} (C)					& \\
	Decomposition theorem 			& Pap--Szeg\H{o} \cite{PS04} 			& Iwata--Takazawa \cite{IT08} \\
	LP formulation 					& Kir\'{a}ly--Makai \cite{KM04} 		& Takazawa \cite{Tak12wief} \\
	Algorithm (Weighted) 			& Takazawa \cite{Tak08} (C)				& Takazawa \cite{Tak12wief} (C)\\\hline
	\end{tabular}

	\end{table}

\subsubsection{Restricted $t$-matching}

The triangle-free $2$-matching and square-free $2$-matching problems are types of 
the \emph{restricted $2$-matching problem}, 
wherein 
a main objective is to provide a tight relaxation of the 
TSP\@. 

Here, let $G=(V,E)$ be a simple undirected graph. 
For $v \in V$, 
let $\delta(v) \subseteq E$ denote the set of edges incident to $v$. 
For a positive integer $t$, 
a vector $x \in \ZZ_+^E$ is called a \emph{$t$-matching} (resp.,\ \emph{$t$-factor}) if 
$\sum_{e \in \delta(v)}x(e) \le t$ (resp.,\ $\sum_{e \in \delta(v)}x(e) = t$) for each $v \in V$. 
A $2$-matching $x$ is called \emph{triangle-free} if it excludes a triple of edges $(e_1,e_2,e_3)$ 
such that 
$e_1$, $e_2$, and $e_3$ form a cycle and $x(e_1)=x(e_2)=x(e_3)=1$. 
For the maximum-weight triangle-free $2$-matching problem, 
a combinatorial algorithm, 
together with a totally dual integral formulation, is designed \cite{CP80,Pap04}.

If we only deal with simple $2$-matchings $x \in \{0,1\}^E$, 
the triangle-free $2$-matching problem becomes much more complicated \cite{HartD}. 
A vector $x \in \{0,1\}^E$ is identified by an edge set $F \subseteq E$ such that $e \in F$ if and only if $x(e)=1$. 
That is, 
an edge set $F \subseteq E$ is called a simple $t$-matching if $|F \cap \delta(v)| \le t$ for each $v \in V$. 
For a positive integer $k$, 
a simple $2$-matching is called \emph{\c{k}-free} if it excludes cycles of length at most $k$. 
Finding a maximum simple \c{k}-free $2$-matching is NP-hard for $k\ge 5$, 
and open for $k=4$. 

In contrast, 
the simple \c{4}-free $2$-matching problem becomes tractable in bipartite graphs. 
We often refer to simple \c{4}-free $2$-matching in a bipartite graph as \emph{square-free $2$-matching}. 
For the square-free $2$-matching problem, 
extensions of the classical matching theory, 
such as 
min-max theorems \cite{Fra03,Hart06,Kir99,Kir09}, 
combinatorial algorithms \cite{Hart06,Pap07}, 
and 
decomposition theorems \cite{Tak17DAM}, 
have been established. 

Further, 
two generalizations of the square-free $2$-matchings have been proposed. 
Frank \cite{Fra03} introduced a generalization, 
\emph{$\ktt$-free $t$-matchings} in bipartite graphs, 
and provided a min-max theorem. 
Another generalization introduced in \cite{Tak17} 
is \emph{\ufeas{} $2$-matchings}. 
For $\U \subseteq 2\sp{V}$, 
a $2$-matching is \emph{\ufeas} if it does not contain a $2$-factor in $U$ for each $U \in \U$. 
Takazawa \cite{Tak17} presented a min-max theorem, 
a combinatorial algorithm, 
and 
decomposition theorems 
for the case where each $U \in \U$ induced a Hamilton-laceable graph \cite{Sim78},
by extending the aforementioned theory for square-free $2$-matchings in bipartite graphs.

	For the weighted case, 
	Kir\'aly \cite{Kir09} proved that finding a maximum-weight square-free $2$-matching is NP-hard (see also \cite{Fra03}). 
	However, 
	Makai \cite{Mak07} presented a linear programming formulation of the weighted $\ktt$-free $t$-matching problem in bipartite graphs 
	with dual integrality for a special case where the weight is \emph{vertex-induced} on each $\ktt$
	(see Definition \ref{DEFvertexinduced}). 
	Takazawa \cite{Tak09} designed a combinatorial algorithm for this case. 
	The assumption on the weight is supported by discrete convexity in \cite{KST12}, 
	which proved that 
	maximum-weight square-free $2$-matchings in bipartite graphs 
	induce an M-concave function on a jump system \cite{Mur06} 
	if and only if the edge weight is vertex-induced on every square.

	The aforementioned results on simple restricted $t$-matchings are for bipartite graphs. 
	We should also mention another graph class where restricted $t$-matchings are tractable, 
	degree-bounded graphs. 
	In subcubic graphs, 
	optimal $2$-matchings excluding the cycles of length three and/or four are tractable \cite{BK12,HL11,HL13,Kob14}. 
	Some of these results are generalized to $t$-matchings 
	excluding $K_{t+1}$ and $K_{t,t}$ in graphs with maximum degree of up to $t+1$ \cite{BV10,KY12}. 
	In bridgeless cubic graphs, 
	there always exists $2$-factors covering all $3$- and $4$-edge cuts \cite{KS08}, 
	and found in polynomial time \cite{BIT13}. 
	A minimum-weight $2$-factor covering all $3$-edge cuts can also be found in polynomial time \cite{BIT13}.

\subsection{Contribution}

It is noteworthy that Pap \cite{Pap07} presented 
combinatorial algorithms for  
the even factor and square-free $2$-matching problems in the same paper. 
These algorithms were based on similar techniques to shrink odd cycles and squares, 
and 
were improved in complexity by Babenko \cite{Bab12}. 
However, 
to the best of our knowledge, 
there has been no comprehensive theory 
that considers both algorithm.

In this paper, 
we discuss \emph{\ufeas{} $t$-matchings} 
(see Definition \ref{DEFufeas}). 
The \ufeas{} $t$-matching problem generalizes not only 
the \ufeas{} $2$-matching problem \cite{Tak17} 
but also 
all of the aforementioned generalizations of the matching problem, 
as well as the TSP and the arborescence problem (see Sect.\ \ref{SECef}). 
The objective of this paper is to provide a unified understanding of these problems. 
One example of such an understanding is 
that $\U$-feasibility is a common generalization of the 
blossom constraint for the nonbipartite matching problem and 
the subtour elimination constraint for the TSP\@. 

The main contributions of this paper are a min-max theorem and an efficient combinatorial algorithm for the maximum \ufeas{} $t$-matching problem in bipartite graphs
under a plausible assumption.  
Note that 
the \ufeas{} $t$-matching problem in \emph{bipartite} graphs can describe the \emph{nonbipartite} matching problem. 
We also remark that it reasonable to impose some assumption in order to 
obtain a tractable class of the \ufeas{} $t$-matching problem. 
(Recall that it can describe the Hamilton cycle problem.)
Indeed, 
we assume that an expanding technique is always valid for the excluded $t$-factors (see Definition \ref{DEFexpansion}). 
This assumption is sufficiently broad to include the instances reduced from nonbipartite matchings, 
even factors in \ocs{} digraphs, 
triangle-free $2$-matchings, 
square-free $2$-matchings, 
and 
arborescences. 
We then show that 
the \emph{$C_{4k+2}$-free $2$-matching problem}, 
a new class of the restricted $2$-matching problem, 
is contained in our framework. 
We prove that 
the $C_{4k+2}$-free $2$-matching problem under a certain assumption 
is described as the \ufeas{} $2$-matching problem under our assumption, 
and thus 
obtain a new class of the restricted $2$-matching problem which can be solved efficiently.

The proposed algorithm commonly extend those for  
nonbipartite matchings, 
even factors, 
triangle-free $2$-matchings, 
square-free $2$-matchings, 
and arborescences. 
Generally, 
the proposed algorithm runs in $\order{t(|V|^3 \alpha + |V|^2 \beta)}$ time, 
where 
$\alpha$ and $\beta$ are the time required to check the feasibility of an edge set 
and expand the shrunk structures, 
respectively. 
The complexities $\alpha$ and $\beta$ are typically small, 
i.e., 
constant or $\order{n}$, 
in the above specific cases (see Sect.\ \ref{SECex}). 

We further establish a linear programming description with dual integrality 
and 
a primal-dual algorithm for the maximum-weight \ufeas{} $t$-matching problem in bipartite graphs. 
The complexity of the algorithm is $\order{t(|V|^3 (|E|+\alpha) + |V|^2 \beta)}$.
For the weighted case, 
we also assume the edge weight to be vertex-induced for each $U \in \U$. 
Note that this assumption 
is also inevitable 
because the maximum-weight square-free $2$-matching problem is NP-hard. 
To be more precise, 
our assumption 
exactly corresponds to the previous assumptions for the maximum-weight even factor and square-free $2$-matching problems, 
both of which are plausible from the discrete convexity perspective \cite{KST12,KT09}. 
This would be an example of a unified understanding of 
even factors and square-free $2$-matchings.

	This paper is organized as follows. 
	In Sect.\ \ref{SECdef}, 
	we present a precise definition of the proposed framework. 
	Sect.\ \ref{SECunweighted} 
	describes 
	a min-max theorem and a combinatorial algorithm for 
	the maximum \ufeas{} $t$-matching problem. 
	In Sect.\ \ref{SECw}, 
	we extend these results to a linear programming formulation with dual integrality and 
	a primal-dual algorithm for the maximum-weight \ufeas{} $t$-matching problem. 
	Conclusions are presented in Sect.\ \ref{SECconcl}. 

\section{Our Framework}
\label{SECdef}

In this section, 
we define the proposed framework 
and 
explain how the previously mentioned problems are reduced. 

\subsection{Optimal $t$-matching Excluding Prescribed $t$-factors}
\label{SECproblem}

Here, 
let $G=(V,E)$ be a simple undirected graph. 
An edge $e$ connecting $u,v \in V$ is denoted by $\{u,v\}$. 
If $G$ is a digraph, 
then 
an arc from $u$ to $v$ is denoted by $(u,v)$. 
For $X \subseteq V$, 
let $G[X] =(X,E[X])$ denote the subgraph of $G$ induced by $X$, 
i.e.,\  
$E[X] = \{\{u,v\} \mid \mbox{$u,v\in X$}, \mbox{$\{u,v\} \in E$}\}$. 
Similarly, 
for $F \subseteq E$, 
define $F[X] = \{\{u,v\} \mid \mbox{$u,v\in X$}, \mbox{$\{u,v\} \in F$}\}$. 
If $X,Y \subseteq V$ are disjoint, 
then $F[X,Y]$ denotes the set of edges in $F$ connecting $X$ and $Y$. 

Recall that 
$\delta(v) \subseteq E$ denotes the set of edges incident to $v \in V$. 
For $F \subseteq E$ and $v \in V$, 
let $\deg_F(v) = |F \cap \delta(v)|$. 
Then, 
$F$ is 
a \emph{$t$-matching} if 
$\deg_F(v) \le t$ for each $v \in V$, 
and a \emph{$t$-factor} if $\deg_F(v) = t$ for every $v \in V$.

\begin{definition}
\label{DEFufeas}
For 
a graph $G=(V,E)$ and 
$\U \subseteq 2^V$, 
a $t$-matching $F \subseteq E$ is called \emph{\ufeas} if 
\begin{align}
\label{EQdefinition}
|F[U]| \le \left\lfloor \frac{t|U|-1}{2} \right\rfloor\end{align} 
for each $U \in \U$. 
\end{definition}

Equivalently, 
a $t$-matching $F$ in $G$ is not \ufeas{} if 
$F[U]$ is a $t$-factor in $G[U]$ for some $U \in \U$. 
This concept is a further generalization of 
the \ufeas{} $2$-matchings introduced in \cite{Tak17}. 

In what follows, 
we consider the maximum \ufeas{} $t$-matching problem,  
whose goal is to find a \ufeas{} $t$-matching $F$ maximizing $|F|$. 
We further deal with the maximum-weight \ufeas{} $t$-matching problem, 
in which the objective is 
to find a \ufeas{} $t$-matching $F$ maximizing $w(F)=\sum_{e \in F}w(e)$ 
for a given edge-weight vector $w \in \RR_+\sp{E}$. 
For a vector $x \in \RR\sp{E}$ and $F \subseteq E$, 
in general we denote $x(F) = \sum_{e \in F}x(e)$. 

In discussing the weighted version, 
we assume that $w$ is \emph{vertex-induced on each $U \in \U$}. 
\begin{definition}
\label{DEFvertexinduced}
For a graph $G=(V,E)$, 
a vertex subset $U \subseteq V$, 
and an edge-weight $w \in \RR\sp{E}$, 
$w$ is called \emph{vertex-induced on $U$} if 
there exists a function $\pi_U: U \to \RR$ on $U$ such that 
$w(\{u,v\}) = \pi_U(u) +\pi_U(v) $ for each $\{u,v\} \in E[U]$. 
\end{definition}

Here, 
as noted previously, 
not only the maximum-weight square-free $2$-matching problem in \emph{bipartite} graphs, 
but also 
many generalizations in \emph{nonbipartite} graphs, 
such as the maximum-weight matching, even factor, and 
triangle-free $2$-matching, 
and 
arborescence problems, 
are 
reduced to the maximum-weight \ufeas{} $t$-matching problem in bipartite graphs 
under the assumption that $w$ is vertex-induced on each $U \in \U$. 
The reduction is shown in Sect.\ \ref{SECef}.  

	\subsection{Special Cases of \ufeas{} $t$-matching in Bipartite Graphs}
	\label{SECef}

	Here we 
	demonstrate how the problems in the literature are reduced to the \ufeas{} $t$-matching problem. 
	In Sect.\ \ref{SECbia}--\ref{SECarb}, 
	we demonstrate reductions to the \ufeas{} $t$-matching problem in \emph{bipartite} graphs, 
	which is the primary focus of this paper. 
	How our algorithm works for those specific cases is described in Sect.\ \ref{SECex}. 
	In Sect.\ \ref{SECnonbi}, 
	we show reductions to 
	the \ufeas{} $t$-matching problem in \emph{nonbipartite} graphs. 
	While we do not discuss solvability in nonbipartite graphs in this paper, 
	we show these reductions in order to demonstrate the generality of the proposed framework. 

	\subsubsection{Restricted $2$-matchings and Hamilton Cycles in Bipartite Graphs}
	\label{SECbia}
	Let $G=(V,E)$ be a simple bipartite graph. 
	If $t=2$ and $\U = \{U \subseteq V \mid |U| = 4\}$, 
	then a \ufeas{} $2$-matching in $G$ is exactly a square-free $2$-matching in $G$. 
	Generally, 
	a simple \c{k}-free $2$-matching in $G$ is exactly a \ufeas{} $2$-matching in $G$ where 
	$\U = \{U \subseteq V \mid 1 \le |U| \le k\}$. 
	For example, 
	if $\U = \{U \subseteq V \mid 1\le |U| \le |V|-1\}$, 
	then 
	the maximum \ufeas{} $2$-matching problem includes the Hamilton cycle problem, 
	i.e.,\ 
	if a maximum \ufeas{} $2$-matching is of size $|V|$, 
	it is a Hamilton cycle. 
	
	Square-free $2$-matchings are generalized to $\ktt$-free $t$-matchings in bipartite graphs \cite{Fra03}. 
	A simple $t$-matching is called \emph{$K_{t,t}$-free} if it does not contain $K_{t,t}$ as a subgraph. 
	A $K_{t,t}$-free $t$-matching in a bipartite graph 
	is exactly a \ufeas{} $t$-matching, where $\U = \{U \subseteq V \mid |U| = 2t\}$.

	\subsubsection{Matchings and Even Factors in Nonbipartite Graphs}
	\label{SECbief}

	First, we show the reduction of the nonbipartite matching problem to the even factor problem. 
	Then, 
	we present the reduction of the even factor problem to the \ufeas{} $t$-matching problem in bipartite graphs. 

	Consider the maximum-weight matching problem in a nonbipartite graph $G=(V,E)$ with weight $w \in \RR\sp{E}$. 
	This 
	can be reduced to the maximum-weight even factor problem 
	in a digraph $D=(V,A)$, 
	where 
	$A=\{ (u,v),(v,u) \mid \{u,v\}\in E \}$, 
	and 
	an arc-weight $w' \in \RR\sp{A}$ is defined by 
	$w'((u,v)) = w'((v,u)) = w(\{u,v\})$. 
	For a matching $M\subseteq E$ in $G$, 
	it is clear that 
	there exists 
	an even factor $F \subseteq A$ in $D$ with $w'(F) = 2w(M)$. 
	Conversely, 
	for an even factor $F \subseteq A$ in $D$, 
	there exists a matching $M \subseteq E$ with $w(M) \ge w'(F)/2$. 

	Here, 
	let $D=(V,A)$ and $w \in \RR^{A}$ be an arbitrary 
	instance of the maximum-weight even factor problem 
	(Fig.\ \ref{FIGef}). 
	Then, 
	define an instance of the maximum-weight \ufeas{} $t$-matching problem as follows. 
	Let 
	$t=1$. 
	For each $u \in V$, 
	let 
	$\p{u}$ and $\m{u}$ be two copies of $u$, 
	and 
	define 
	$\p{\hat{V}} =\{\p{u} \mid u \in V\}$ and $\m{\hat{V}} =\{\m{u} \mid u \in V\}$. 
	For $U \subseteq V$, 
	denote $\hat{U} = \bigcup_{u \in U}\{\p{u},\m{u}\}$. 
	Now define a bipartite graph $\hat{G}=(\hat{V},\hat{E})$, 
	$\hat{\U} \subseteq 2^{\hat{V}}$, 
	and 
	an edge-weight $\hat{w} \in \RR^{\hat{E}}$ 
	by
	\begin{align}
	&{}\hat{E} = \{\{\p{u},\m{v}\} \mid (u,v) \in A\}, 
	\quad
	\hat{\U} = \{ \hat{U} \mid \mbox{$U \subseteq V$}, \mbox{$|U|$ is odd}\}, 
	\label{EQef}\\
	&{}
	\hat{w}(\{\p{u},\m{v}\}) = w((u,v)) \quad (\{\p{u},\m{v}\} \in \hat{E}).
	\end{align}

	Note that a $1$-matching in $\hat{G}$ corresponds to a path-cycle factor in $D$. 
	For $\hat{U}\in \hat{\U}$, 
	a $1$-factor in $\hat{G}[\hat{U}]$ corresponds to 
	a vertex-disjoint collection of cycles through $U$ in $D$, 
	which should contain at least one odd cycle. 
	Thus, 
	$\hat{\U}$-feasibility of a $1$-matching in $\hat{G}$ exactly corresponds to 
	excluding odd cycles in a path-cycle factor in $D$, 
	which results in an even factor. 

	If $(D, w)$ is \ocs, 
	$\hat{G}[\hat{U}]$ is a symmetric bipartite graph and 
	$\hat{w}$ is vertex-induced on $\hat{U}$ 
	for each $\hat{U} \in \hat{\U}$. 
	Thus, 
	the instance constructed in the reduction satisfies 
	our assumption 
	that $\hat{w}$ is vertex-induced on each $\hat{U} \in \hat{\U}$. 
	
	\begin{figure}
	\centering
	\includegraphics[height=.25\textheight]{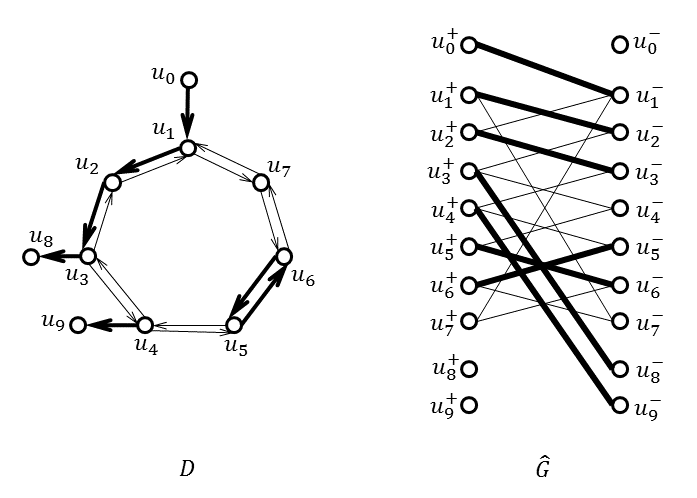}
	\caption{The maximum even factor problem in $D$ is reduced to the maximum $\hat{\U}$-feasible 1-matching problem in $\hat{G}$, 
	where $\hat{\U} = \{ \hat{U} \mid \mbox{$U \subseteq V$}, \mbox{$|U|$ is odd}\}$. 
	The set of thick arcs in $D$ is an even factor 
	that corresponds to the set of thick edges in $\hat{G}$, 
	which is a $\hat{\U}$-feasible $1$-matching.}
	\label{FIGef}
	\end{figure}

	\subsubsection{Triangle-free $2$-matchings in Nonbipartite Graphs}
	\label{SECtri}
	Here, 
	let $G=(V,E)$ be an undirected nonbipartite graph 
	and $w \in \RR_+^E$. 
	Now 
	define $\p{\hat{V}}$, $\m{\hat{V}}$, and $\hat{U}$ 
	as described in Sect.\ \ref{SECbief}. 
	Let $t=1$ 
	and 
	define $\hat{G} =(\hat{V}, \hat{E})$, 
	$\hat{\U} \subseteq 2\sp{\hat{V}}$, 
	and $\hat{w}\in \RR\sp{\hat{E}}$ 
	by 
	\begin{align*}
	&{}\p{\hat{V}} =\{\p{u} \mid u \in V\},\qquad \m{\hat{V}} =\{\m{v} \mid v \in V\}, \\
	&\hat{E} = \bigcup_{\{u,v\} \in E}\{ \{\p{u},\m{v}\}, \{\p{v},\m{u}\} \}, \qquad
	\hat{\U} = \{ \hat{U} \mid \mbox{$U \subseteq V$, $|U| =3$} \},\\
	&\hat{w}(\{\p{u},\m{v}\})=\hat{w}(\{\p{v},\m{u}\}) = w(\{u,v\}) \quad (\{u,v\} \in E). 
	\end{align*}
	It is straightforward that 
	a triangle-free $2$-matching $F$ in $G$ corresponds to a $\hat{\U}$-feasible $1$-matching $\hat{F}$ in $\hat{G}$ such that $w(F) = \hat{w}(\hat{F})$, 
	and vice versa (Fig.\ \ref{FIGtriangle}). 
	
	\begin{figure}
	\centering
	\includegraphics[height=.18\textheight]{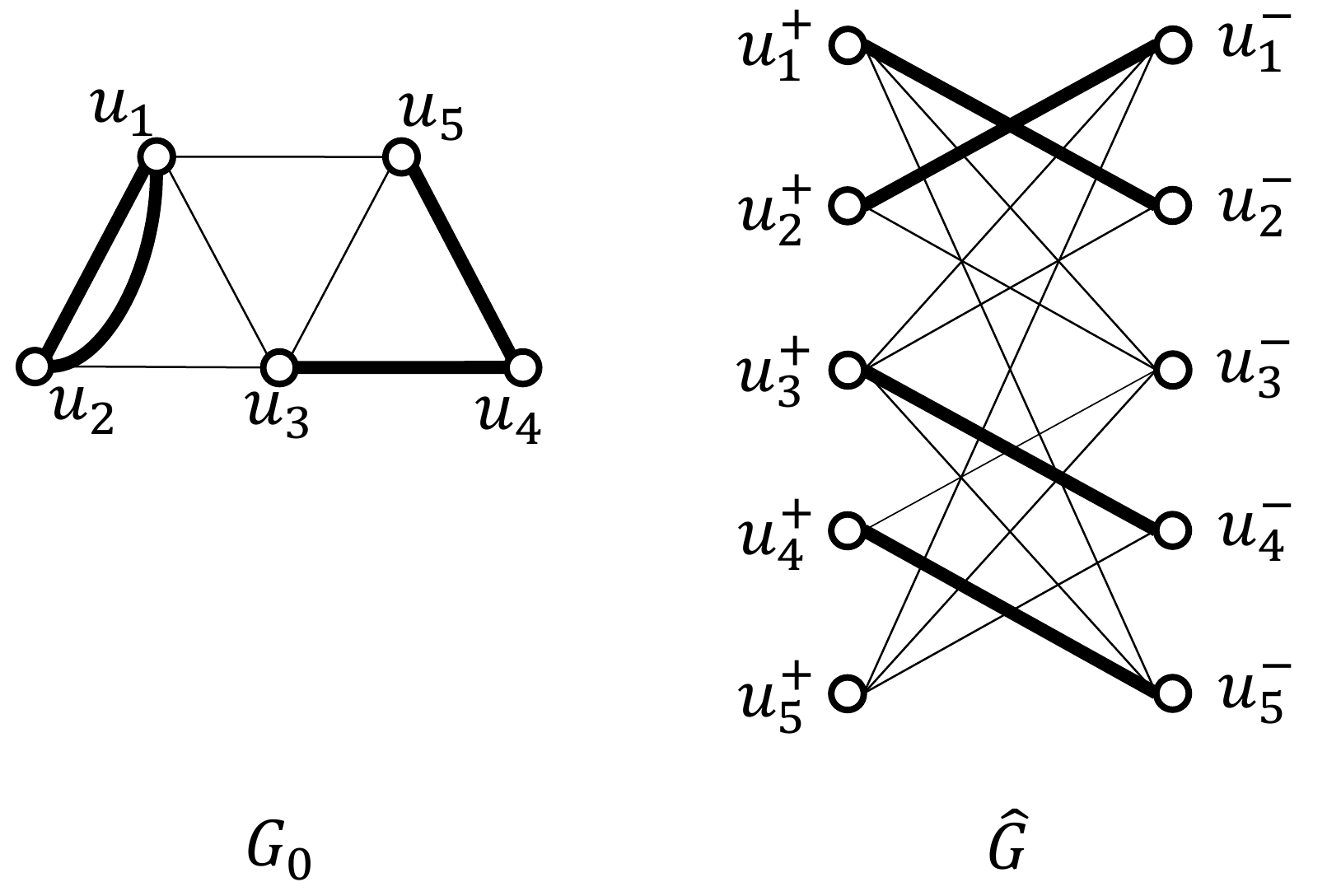}
	\caption{The maximum triangle-free $2$-matching problem in $G$ is reduced to the maximum $\hat{\U}$-feasible 1-matching problem 
	in $\hat{G}$, 
	where $\hat{\U} = \{ \hat{U} \mid \mbox{$U \subseteq V$}, \mbox{$|U|=3$}\}$. 
	The set of thick edges in $G$ is a triangle-free $2$-matching that 
	corresponds to the set of thick edges in $\hat{G}$, 
	which is a $\hat{\U}$-feasible $1$-matching.}
	\label{FIGtriangle}
	\end{figure}

	\subsubsection{Matroids and Arborescences}
	\label{SECarb}
	Here, 
	let $\matroid$ be a matroid with ground set $V$ 
	and circuit family $\Cir \subseteq 2\sp{V}$. 
	The problem of finding a maximum-weight 
	independent set in $\Matroid$ with respect to $w \in \RR\sp{V}$ is described as 
	the maximum-weight $\hat{\U}$-feasible $t$-matching problem 
	in a bipartite graph $\hat{G}$ as follows. 
	Let $t=1$. 
	Define $\hat{G} =(\hat{V}, \hat{E})$, 
	$\U \subseteq 2\sp{\hat{V}}$, 
	and $\hat{w}\in \RR\sp{\hat{E}}$ 
	by 
	\begin{align*}
	&{}\hat{V}^+ = \{ \p{u} \mid u \in V \}, \quad  
	\hat{V}^- = \{ \m{v} \mid v \in V \}, \\
	&{}\hat{E} = \{ \{\p{v}, \m{v}\} \mid v \in V \}, \quad
	\hat{\U} = \left\{ \bigcup_{v \in C} \{\p{v}, \m{v}\} \mathrel{}\middle|\mathrel{} C \in \cir \right\}, \\
	&{}\hat{w}(\{\p{v}, \m{v}\}) = w(v). 
	\end{align*}
	Then, it is straightforward that 
	$I \subseteq V$ is an independent set in $\matroid$ 
	if and only if 
	the edge set $\{ \{\p{v}, \m{v}\} \mid v \in I \}$ is a 
	$\hat{\U}$-feasible $1$-matching in $\hat{G}$. 
	
	Arborescences in a digraph are a special case of matroid intersection. 
	Although we do not know how to describe matroid intersection in our framework, 
	the arborescence problem 
	cab be reduced to the 
	\ufeas{} problem in a bipartite graph as follows. 
	Let $D=(V,A)$ 
	be a digraph in which we are asked to find a maximum-weight arborescence 
	with respect to an arc weight $w\in \RR\sp{A}$. 
	Let $t=1$, 
	and 
	define $\hat{G} =(\hat{V}, \hat{E})$, 
	$\hat{\U} \subseteq 2\sp{\hat{V}}$, 
	and $\hat{w}\in \RR\sp{\hat{E}}$ 
	by 
	\begin{align*}
	&{}\hat{V}^+ = \{ \p{a} \mid a \in A \}, \quad  
	\hat{V}^- = \{ \m{v} \mid v \in V \}, \quad
	\hat{E} = \{ \{\p{a}, \m{v}\} \mid \mbox{$v$ is the head of $a$ in $D$} \}, \\
	&{}\hat{\U} = \left\{ \{\p{a} \mid a \in A(C) \} \cup \{\m{v} \mid v \in V(C) \} \mid \mbox{$C$ is a directed cycle in $D$} \right\}, \\
	&{}\hat{w}(\{\p{a}, \m{v}\}) = w(a),
	\end{align*}
	where $A(C)$ and $V(C)$ denote the sets of arcs and vertices of a directed cycle $C$, 
	respectively. 
	Again, it is straightforward that 
	$A' \subseteq A$ is an arborescence in $D$ 
	if and only if 
	the edge set $$\{ \{\p{a}, \m{v}\} \mid \mbox{$a \in A'$, $v$ is the head of $a$ in $D$} \}$$ 
	is a 
	$\hat{\U}$-feasible $1$-matching in $\hat{G}$. 
	
	\subsubsection{Special Cases of Nonbipartite \ufeas{} $t$-matchings}
	\label{SECnonbi}
	The simple \c{k}-free $2$-matching problem in a nonbipartite graph $G=(V,E)$ 
	is exactly the \ufeas{} $2$-matching problem in $G$, 
	where 
	$\U = \{U \subseteq V \mid 1\le |U| \le k\}$. 
	For example, 
	if $\U = \{U \subseteq V \mid 1\le |U| \le|V|-1\}$, 
	then 
	a \ufeas{} $2$-matching of size $|V|$ is 
	exactly a Hamilton cycle. 

	The $K_{t+1}$-free $t$-matching problem \cite{BV10} 
	is a generalization of the simple triangle-free $2$-matching problem. 
	A simple $t$-matching is called \emph{$K_{t+1}$-free} 
	if it does not contain $K_{t+1}$ as a subgraph. 
	Now 
	a $K_{t+1}$-free $t$-matching is exactly 
	a \ufeas{} $t$-matching, where $\U=\{U\subseteq V \mid |U|= t+1\}$. 
	
	A \emph{$2$-factor covering prescribed edge cuts} is also described as a \ufeas{} $2$-matching. 
	Here, 
	let $C \subseteq E$ be an edge cut, 
	i.e., 
	$C$ is an inclusion-minimal edge subset such that 
	deleting $C$ makes $G$ disconnected. 
	A $2$-factor $F$ covers $C$ if $F \cap C \neq \emptyset$. 
	For a family $\mathcal{C}$ of edge cuts, 
	a $2$-factor covering every edge cut in $\mathcal{C}$ is a relaxed concept of Hamilton cycles, 
	i.e.,\ 
	if $\mathcal{C}$ is the family of all edge cuts in $G$, 
	then a $2$-factor covering all edge cuts in $\mathcal{C}$ is a Hamilton cycle. 
	
	Now a $2$-factor covering every edge cuts in $\mathcal{C}$ is described as a \ufeas{} $2$-factor by 
	putting 
	$\mathcal{U} = \{ U \subseteq V \mid \mbox{$\delta(U) \in \mathcal{C}$} \}$, 
	where 
	$\delta(U)$ denotes $E[U, V \setminus U]$,
	i.e.,\ 
	the set of edges connecting $U$ and $V \setminus U$. 
	For example, 
	a $2$-factor covering all $3$- and $4$-edge cuts \cite{BIT13,KS08} is a \ufeas{} $2$-factor, 
	where $\U =\{ U \subseteq V \mid \mbox{$\delta(U)$ is a $3$-edge cut or a $4$-edge cut}  \}$.

\section{Maximum \ufeas{} $t$-matching}
\label{SECunweighted}

In this section, 
we present a min-max theorem and a combinatorial algorithm 
for the maximum \ufeas{} $t$-matching problem in bipartite graphs. 
The proposed algorithm commonly extends those for 
nonbipartite matchings \cite{Edm65}, 
even factors \cite{Pap07}, 
triangle-free $2$-matchings \cite{CP80}, 
square-free $2$-matchings \cite{Hart06,Pap07}, 
and 
arborescences \cite{CL65,Edm67}. 
We begin with a weak duality theorem in Sect.\ \ref{SECweak}. 
The proposed algorithm is described 
in Sect.\ \ref{SECalg}, 
and 
its validity is proved in Sect.\ \ref{SECminmax} 
together with 
the min-max theorem (strong duality theorem). 
In Sect.\ \ref{SECex},
we demonstrate 
how the algorithm works in such special cases.

\subsection{Weak Duality}
\label{SECweak}

Here, let $G=(V,E)$ be an undirected graph 
and 
let 
$\U \subseteq 2\sp{V}$. 
For weak duality, 
$G$ does not need to be bipartite. 
For $X \subseteq V$, 
define 
$\U_X \subseteq \U$ 
and 
$C_X \subseteq X$ by 
\begin{align*}
&{}\U_X = \{ U \in \U \mid \mbox{$U$ forms a component in $G[X]$} \}, &
&{}C_X = X \sm \bigcup_{U \in \U_X}U. 
\end{align*}
Then,  
the following inequality holds for an arbitrary \ufeas{} $t$-matching $F \subseteq E$ and $X \subseteq V$. 
\begin{lemma}
\label{LEMweak}
Let 
$G=(V,E)$ be an undirected graph, 
$\U \subseteq 2\sp{V}$, 
and 
$t$ be a positive integer. 
For an arbitrary \ufeas{} $t$-matching $F \subseteq E$ 
and $X \subseteq V$, 
it holds that 
\begin{align}
\label{EQweak}
|F| \le t|X| + |E[C_{V\sm {X}}]| + \sum_{U \in \U_{V \sm X}}\left\lfloor \frac{t|U|-1}{2} \right\rfloor. 
\end{align}
\end{lemma}

\begin{proof}
By counting the number of edges in $F $ incident to $X$, 
we obtain 
\begin{align}
\label{EQsaturated}
&{}2|F[X]| + |F[X, V \sm {X}]| \le t|X|.
\end{align}
In $G[V \setminus X]$, 
it holds that 
\begin{align}
\label{EQcritical}
&{} |F[V \sm {X}]| \le |E[C_{V \sm {X}}]| + \sum_{U \in \U_{V \sm X}}\left\lfloor \frac{t|U|-1}{2} \right\rfloor. 
\end{align}
By summing \eqref{EQsaturated} and \eqref{EQcritical}, 
we obtain 
\begin{align}
\label{EQsumup}
&{}2|F[X]| + |F[X, V \sm {X}]| + |F[V \sm {X}]| \le  t|X| + |E[C_{V \sm {X}}]| + \sum_{U \in \U_{V \sm X}}\left\lfloor \frac{t|U|-1}{2} \right\rfloor. 
\end{align}
Since $|F| = |F[X]| + |F[X, V \sm {X}]| + |F[V \sm {X}]|$ is at most the left-hand side of \eqref{EQsumup}, 
we obtain \eqref{EQweak}. 
\end{proof}

\subsection{Algorithm}
\label{SECalg}

Hereafter, 
we assume that $G$ is bipartite. 
Let $G=(V,E)$ be a simple undirected bipartite graph. 
Here, 
we denote the two color classes of $V$ by $\p{V}$ and $\m{V}$. 
For $X \subseteq V$, 
denote $\p{X} = X \cap \p{V}$ and $\m{X} = X \cap \m{V}$. 
The endvertices of an edge $e\in E$ in $\p{V}$ and $\m{V}$ are denoted by $\partial^+ e$ and $\partial^- e$, respectively.

We begin by describing the shrinking of a forbidden structure $U \in \U$.  
For concise notation, 
we denote the input graph as $\h{G} = (\h{V}, \h{E})$ 
and 
the graph generated by potential repeated shrinkings 
as $G=(V,E)$. 
Consequently, 
we have $\U \subseteq 2\sp{\hat{V}}$. 
The solution in hand is denoted by $F \subseteq E$. 

Intuitively, 
shrinking of $U$ involves
identifying all vertices in $\p{U}$ and $\m{U}$ to obtain new vertices $\p{u_U}$ and $\m{v_U}$, respectively, 
and 
deleting all edges in $E[U]$. 
In a shrunk graph $G=(V,E)$, 
we refer to a vertex $v \in V$ as a \emph{natural vertex} if $v$ is a vertex in the original graph $\h{G}$, 
and as a \emph{pseudovertex} if it is a newly added vertex when shrinking some $U \in \U$. 
We denote the set of natural vertices as $\vo$, 
and the set of pseudovertices as $\vp$. 
For $X \subseteq \h{V}$, 
define 
$\xo = X \cap \vo$ 
and 
$\xp = \bigcup\{\p{u_U}, \m{v_U} \mid \mbox{$\p{u_U}, \m{v_U} \in \vp$, $U \cap X \neq \emptyset$} \}$.  
For $X \subseteq V$, 
define $\hat{X}\subseteq \hat{V}$ by 
$\hat{X} = \xn \cup \bigcup \{\p{U} \mid \p{u_U}\in  X \cap \vp \} \cup \bigcup \{\m{U} \mid \m{v_U}\in  X \cap \vp \}$.

A formal description of shrinking $U \in \U$ is given as follows. 

\paragraph{Procedure $\mbox{\textsc{Shrink}}(U)$.}
Let $\p{u_U}$ and $\m{v_U}$ be new vertices, 
and  
reset the endvertices of an edge $e \in E \sm E[\un \cup \up]$ 
with 
$\p{\partial}e = u$ and 
$\m{\partial}e = v$ by 
\begin{align*}
{}&{}\p{\partial}e := \p{u_U} \quad \mbox{if $u \in \p{\un} \cup \p{\up}$}, \\
{}&{}\m{\partial}e := \m{v_U} \quad \mbox{if $v \in \m{\un} \cup \m{\up}$}. 
\end{align*}
Then, 
update $G$ by 
\begin{align*}
	&{}\p{V} := (\p{V} \sm (\p{\un} \cup \p{\up})) \cup \{\p{u_U}\}, &
	&{}\m{V} := (\m{V} \sm (\m{\un} \cup \m{\up})) \cup \{\m{v_U}\}, &
	&{}E := E \sm E[U].  
	\end{align*}
Finally, 
$F := F \cap E$  
and 
return $(G,F)$. 

\medskip

Procedure $\mbox{\textsc{Expand}}(G,F)$ is to execute 
the reverse of $\mbox{\textsc{Shrink}}(U)$ for all shrunk $U \in \U$ 
to obtain the original graph $\hat{G}$. 
Here, 
a key point is that $\lfloor (t|U|-1)/2\rfloor$ edges are added to $F$ from $\hat{E}[U]$ for each $U \in \U$.

\paragraph{Procedure $\mbox{\textsc{Expand}}(G,F)$.}
Let $G := \h{G}$. 
For each inclusionwise maximal $U\in \U$ that is shrunk, 
we add 
$F_U \subseteq \hat{E}[U]$ of $\lfloor (t|U|-1)/2\rfloor$ edges 
to $F$ such that $F$ is a \ufeas{} $t$-matching in $\h{G}$. 
Then return $(G,F)$. 

\medskip

In Procedure $\mbox{\textsc{Expand}}(G,F)$, 
the existence of $F_U$  is non-trivial. 
To attain that 
$\hat{F}=F \cup \bigcup\{ F_U \mid \mbox{$U \in \U$ is a maximal shrunk set}\}$ 
is a $t$-matching in $\hat{G}$, 
it should be satisfied for 
$F \subseteq E$ and 
$F_U \subseteq \hat{E}[U]$ 
that 
\begin{align}
\label{EQdegconst}
&\deg_F(u) \le 
\begin{cases} 
	t & (u \in \vo), \\
	1 & (u \in \vp)
\end{cases}\\
\label{EQexpandt}
&\deg_{F_U}(u) 
	\begin{cases}
	=t-1 		& (\mbox{$u$ is incident to an edge in $F[U, V \sm U]$}), \\
	\le t 		& (\mbox{otherwise}). 
	\end{cases}
\end{align}
To achieve this, 
we maintain that 
$F$ satisfies the degree constraint \eqref{EQdegconst}. 
Moreover, 
we assume that, 
for an arbitrary $F$ with \eqref{EQdegconst}, 
there exists $F_U$ 
satisfying $|F_U| = \lfloor (t|U|-1)/2\rfloor$ 
and \eqref{EQexpandt} 
for every maximal shrunk set $U \in \U$. 
This assumption is formally defined as follows.

\begin{definition}
\label{DEFexpansion}
Let $\hat{G}=(\hat{V}, \hat{E})$ be a bipartite graph, 
$\U \subseteq 2\sp{\hat{V}}$, 
and 
$t$ be a positive integer. 
For pairwise disjoint $U_1,\ldots, U_l \in \U$, 
let $G=(V,E)$ denote the graph obtained from $\hat{G}$ by executing $\mbox{\textsc{Shrink}}(U_1)$, \ldots, $\mbox{\textsc{Shrink}}(U_l)$, 
and 
let $F \subseteq E$ be an arbitrary edge set satisfying \eqref{EQdegconst}. 
If there exists $F_{U_i} \subseteq \hat{E}[U_i]$ satisfying $|F_{U_i}| = t|U_i|/2-1$ and \eqref{EQexpandt} for each $i=1,\ldots, l$, 
we say that \emph{$(\hat{G}, \U,t)$ admits expansion}. 
\end{definition}

In what follows, 
we assume that 
$(\hat{G},\U,t)$ admits expansion. 
This is exactly the class of $(\hat{G},\U,t)$ to which our algorithm is applicable. 

This assumption and the degree constraint \eqref{EQdegconst} guarantee that 
we can always obtain a $t$-matching $\hat{F}=F \cup \bigcup\{ F_U \mid \mbox{$U \in \U$ is a maximal shrunk set}\}$ in $\hat{G}$. 
Furthermore, 
we should consider the $\U$-feasibility of $\hat{F}$. 
We refer to $F$ in $G$ as \emph{feasible} if $\hat{F}$ is \ufeas{}.  
If there are several possibilities of $F_U$, 
we say that $F$ is \ufeas{} if there is at least one \ufeas{} $\hat{F}$. 
In other words, 
$F$ satisfying \eqref{EQdegconst} is not feasible 
if,  
for any possibility of $\hat{F}$, 
\begin{align}
|\hat{F}[U']| = \frac{t|U'|}{2}
\label{EQviolate}
\end{align}
holds for some $U' \subseteq \U$, 
and 
$\hat{F}$ will have a $t$-factor in $\hat{G}[U']$. 

See Fig.\ \ref{FIGfeas} for an example. 
Here, $t=2$, and we expand $U=\{u_1,u_2,u_3,u_4,v_1,v_2,v_3,v_4\}$ 
by adding $F_U$ of $\lfloor (t|U|-1)/2\rfloor = 7$ edges satisfying \eqref{EQexpandt}. 
However, 
$\hat{F}_1$ is \emph{not} \ufeas{} because 
it violates \eqref{EQdefinition} for $U'=\{u_3,u_4,u_5,u_6,v_3,v_4,v_5,v_6\}$. 
On the other hand, 
$\hat{F}_2$ 
satisfies \eqref{EQdefinition} for $U'$. 
Thus, 
$F$ in $G$ is called feasible, and 
we select $\hat{F}_2$ when expanding $U$.

\begin{figure}
\centering
\includegraphics[height=.25\textheight]{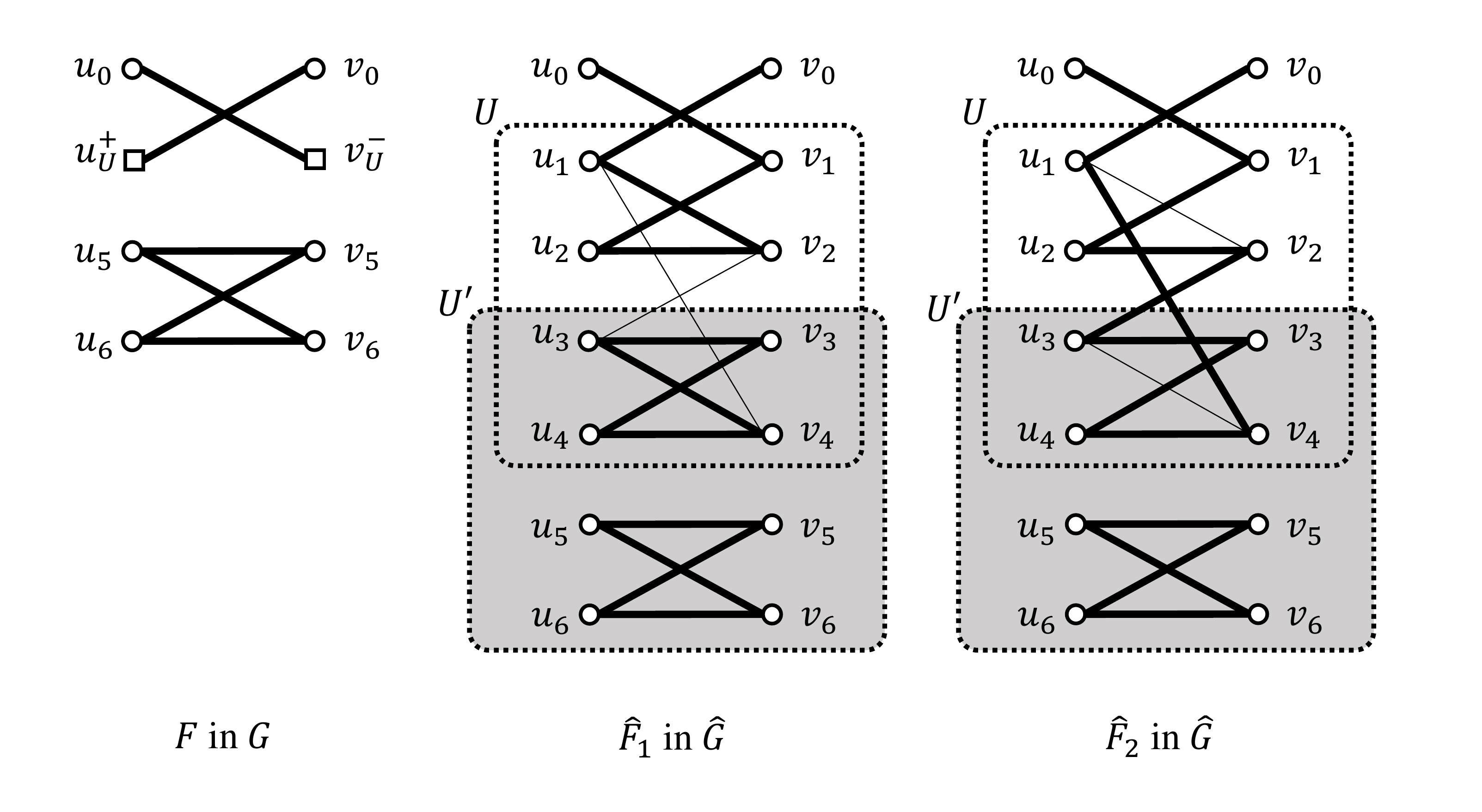}
\caption{In expanding $U\in \U$, $\hat{F}_1$ is inappropriate because it contains a $t$-factor in $\hat{G}[U']$, 
while $\hat{F}_2$ is appropriate.}
\label{FIGfeas}
\end{figure}

Here, 
we describe our algorithm in detail. 
The algorithm begins with $G=\h{G}$ and an arbitrary \ufeas{} $t$-matching $F \subseteq \h{E}$, 
typically $F = \emptyset$. 
We first construct an auxiliary digraph. 

\paragraph{Procedure $\mbox{\textsc{ConstructAuxiliaryDigraph}}(G,F)$.}
Construct a digraph $(V,A)$ 
defined by 
$$A = \{ (u,v) \mid \mbox{$u\in \p{V}$, $v \in \m{V}$, $\{u,v\} \in E \sm F$}   \}
				\cup \{ (v,u) \mid \mbox{$u\in \p{V}$, $v \in \m{V}$, $\{u,v\} \in F$}   \}.$$
Define the sets of source vertices $S \subseteq \p{V}$ and 
sink vertices $T \subseteq \m{V}$ by 
\begin{align*}
&S = 	\{ u \in \p{\vo} \mid \deg_F(u) \le t-1 \} \cup \{ \p{u_U} \in \p{\vp} \mid \deg_F(\p{u_U}) = 0 \},\\
&T = 	\{ v \in \m{\vo} \mid \deg_F(v) \le t-1 \} \cup \{ \m{v_U} \in \m{\vp} \mid \deg_F(\m{v_U}) = 0 \}. 
\end{align*}
Then, return $D=(V,A;S,T)$.

\medskip

Suppose that there exists a directed path $P=(e_1,f_1,\ldots, e_l, f_l, e_{l+1})$ in $D$ from $S$ to $T$. 
Note that 
$e_i \in E \sm F$ ($i=1,\ldots, l+1$) 
and ${f}_i \in F$ ($i=1,\ldots, l$). 
We denote the symmetric difference $(F \sm P) \cup (P \sm F)$ of $F$ and $P$ by $F \sd P$. 
If $F\sd P$ is feasible, 
we execute $\mbox{\textsc{Augment}}(G,F,P)$ below. 
We then execute 
\textsc{Expand}$(G,F)$.

\paragraph{Procedure $\mbox{\textsc{Augment}}(G,F,P)$.}
Let $F:=F\sd P$ and  
return $F$.  

\medskip

If $F\sd P$ is not feasible, 
we apply 
$\mbox{\textsc{Shrink}}(U)$ 
after determining a set $U \in \U$ to be shrunk by the following procedure.

\paragraph{Procedure $\mbox{\textsc{FindViolatingSet}}(G,F,P)$.}
For $i=1,\ldots, l$, 
define 
$F_i = (F\sm\{f_1,\ldots f_i\}) \cup\{e_1,\ldots, e_{i}\}$. 
Also define $F_0 = F$ and $F_{l+1} = F\sd P$. 
Let $i^*$ be the minimum index $i$ such that $F_i$ is not feasible, 
and 
let $U \in \U$ satisfy 
\eqref{EQviolate} for $F = F_{i^*}$. 
Then, 
let 
$F := F_{i^*-1}$, 
and return $(F,U)$. 

\medskip

Finally, 
if $D$ does not have a directed path from $S$ to $T$, 
we determine the minimizer ${X} \subseteq \h{V}$ of \eqref{EQmin} as follows. 

\paragraph{Procedure $\mbox{\textsc{FindMinimizer}}(G,F)$.}
Let $R\subseteq V$ be the set of vertices reachable from $S$, 
and let 
$X := (\p{V} \sm \p{R}) \cup \m{R}$. 
If a natural vertex $v \in \m{V} \sm X$ has $t$ edges in $F$ connecting $\p{R}$ and $v$, 
then 
$X := X \cup \{v\}$. 
If a pseudovertex $\m{v}_U \in \m{V} \sm X$ has one edge in $F$ connecting $\p{R}$ and $\m{v}_U$, 
then 
$X := X \cup \{\m{v}_U\}$. 
Finally, 
return $X:=\hat{X}$. 

\medskip

We then 
apply $\mbox{\textsc{Expand}}(G,F)$ and 
the algorithm terminates by returning $F \subseteq \hat{E}$ and $X \subseteq \hat{V}$. 

Now the description of the algorithm is completed. 
The pseudocode of the algorithm is presented in Algorithm \ref{ALGunweighted}. 
The optimality of $F$ and $X$ is proved 
in Sect.\ \ref{SECminmax}. 
We exhibit how the algorithm works in specific cases such as 
the square-free $2$-matching, even factor,
triangle-free $2$-matching, 
and arborescence problems, 
and discuss the complexity for these cases in Sect.\ \ref{SECex}. 
Before that, 
we analyze the complexity of the algorithm for the general case. 

	\begin{algorithm}[t]
	\caption{Maximum \ufeas{} $t$-matching}
	\label{ALGunweighted}
	\begin{algorithmic}[1]
	\State $G \leftarrow \hat{G}$, $F \leftarrow \emptyset$, $D \leftarrow \mbox{\textsc{AuxiliaryDigraph}}(G,F)$ 
	\While{$D$ has an $S$-$T$ path $P$}
			\If{$F\sd P$ is feasible}
				\State $F \leftarrow \mbox{\textsc{Augment}}(G,F,P)$
				\State $(G,F) \leftarrow \mbox{\textsc{Expand}}(G,F)$
				\State $D \leftarrow \mbox{\textsc{AuxiliaryDigraph}}(G,F)$ 
			\Else
				\State $(F,U) \leftarrow \mbox{\textsc{ViolatingSet}}(G,F,P)$
				\State $(G,F) \leftarrow \mbox{\textsc{Shrink}}(U)$
				\State $D \leftarrow \mbox{\textsc{AuxiliaryDigraph}}(G,F)$ 
			\EndIf
	\EndWhile
	\State $X \leftarrow \mbox{\textsc{Minimizer}}(G,F)$, 
	$(G,F) \leftarrow \mbox{\textsc{Expand}}(G,F)$\\
	\Return $(F,X)$
	\end{algorithmic}
	\end{algorithm}

Here, 
let $n = |\hat{V}|$ and $m=|\hat{E}|$. 
The complexity of the algorithm varies according to the structure of $(G,\U,t)$. 
Recall that 
$\alpha$ denotes the time required to determine the feasibility of $F$, 
and $\beta$ denotes the time requited to expand $U$.  
To be precise, 
$\alpha$ is the time required to check whether $F$ in a shrunk graph $G$ is feasible, 
and, 
if not, 
find $U \in \U$ for which $F$ satisfies \eqref{EQviolate}. 

Between augmentations, 
we execute $\mbox{\textsc{Shrink}}(U)$ $\order{n}$ times. 
For one $\mbox{\textsc{Shrink}}(U)$, 
we check the feasibility of $F_1,\ldots, F_l$, 
which requires $\order{n\alpha}$ time. 
We then reconstruct the auxiliary digraph. 
Here, we should only update the vertices and arcs on the $S$-$T$ path, 
which takes $\order{n}$ time. 
After augmentation, 
we expand the shrunk vertex sets, 
which takes $\order{n\beta}$ time in total. 
Therefore, 
the complexity for one augmentation is $\order{n^2 \alpha + n\beta}$. 
Since augmentation occurs at most $tn/2$ time, 
the total complexity of the algorithm is $\order{t(n^3\alpha + n^2 \beta)}$.

\subsection{Min-max Theorem: Strong Duality}
\label{SECminmax}

In this section, 
we strengthen Lemma \ref{LEMweak} to be a min-max relation 
and 
prove the validity of Algorithm \ref{ALGunweighted}. 
We show that 
the output $(F,X)$ of the algorithm satisfies \eqref{EQweak} with equality. 
This constructively proves the min-max relation 
for the class of $(G,\U,t)$ that admits expansion.

\begin{theorem}
\label{THminmax}
Let $G=(V,E)$ be a bipartite graph, 
$\U \subseteq 2\sp{V}$, 
and $t$ be a positive integer 
such that $(G,\U,t)$ admits expansion. 
Then, 
the maximum size of a \ufeas{} $t$-matching is equal to 
the minimum of 
\begin{align}
\label{EQmin}
t|X| + |E[C_{V \sm {X}}]| + \sum_{U \in \U_{V \sm X}}\left\lfloor \frac{t|U|}{2} - 1 \right\rfloor , 
\end{align}
where $X$ runs over all subsets of $V$. 
\end{theorem}

\begin{proof}
We denote the output of Algorithm \ref{ALGunweighted} by $(\hat{F},\hat{X}) $. 
Here, 
it is sufficient to 
prove that \eqref{EQsaturated} and \eqref{EQcritical} hold by equality 
for $(\hat{F},\hat{X})$. 
Since $X$ is defined based on reachability in the auxiliary digraph $D$, 
it is straightforward that $\hat{F}[\hat{X}] = \emptyset$. 
Moreover, 
for every $v \in \hat{X}$, 
$\deg_{\hat{F}}(v) = t$ holds; 
thus, \eqref{EQsaturated} holds by equality. 
Finally, 
edges in 
$\h{G}[\h{V}\sm \hat{X}]$ are in $F$ before the last \textsc{Expand}$(G,F)$ or 
are obtained by expanding pseudovertices $\p{u_U}$ and $\m{v_U}$, 
which are isolated vertices in $G[V \sm {X}]$. 
This means that $U$ forms a component in $\h{G}[{\hat{X}}]$; 
thus, the equality in \eqref{EQcritical} holds. 
\end{proof}

\subsection{Applying our Algorithm to Special Cases}
\label{SECex}

Here, we demonstrate how Algorithm \ref{ALGunweighted} is applied to the special cases of 
\c{k}-free $2$-matchings, 
even factors (including nonbipartite matchings), 
triangle-free $2$-matchings, 
and 
arborescences. 
Differences appear in 
determining feasibility in the shrunk graph and 
the edges to be added by expansion. 

\subsubsection{$C_{\le k}$-free 2-matchings in Bipartite Graphs}
The case where $k=4$, 
i.e., square-free $2$-matchings in a simple bipartite graph, 
is the most straightforward example. 
In this case, 
the family of the shrunk vertex sets never becomes nested, 
i.e.,\ 
\textsc{Shrink}($U$) is always applied to a cycle of length four comprising four natural vertices. 
Thus, 
an edge set $F \subseteq E$ is feasible if and only if $F$ excludes a cycle of length four, 
even if the graph is obtained by repeated shrinking. 
Furthermore, 
the feasibility of each $F_i$ ($i=1,\ldots, l$) can be checked in constant time 
because it is sufficient to determine whether the new edge $e_i$ added to $F_i$ is in a square. 

When expanding $U \in \U$, 
it suffices to 
choose $F_U$ consisting of three edges in $E[U]$ and 
satisfying \eqref{EQexpandt} for $t=2$. 
This always yields the $\U$-feasibility of $\hat{F}$ 
and can be performed in constant time for one square. 

For the case where $k \ge 6$, 
the problem becomes more involved. 
Suppose that 
$\U=\{U \subseteq \hat{V} \mid 1 \le |U| \le 6\}$ and 
we expand $U \in \U$ with $|U|=6$. 
We then select $F_U\subseteq E[U]$ with $|F_U| = 5$  
according to \eqref{EQexpandt}; 
however, such $F$ might not exist. 
Moreover, 
even if such $F_U$ is found, 
$F_U$ might contain a cycle of length four, 
which violates $\U$-feasibility. 

Such difficulty is inevitable because 
the simple \c{k}-free $2$-matching problem in bipartite graphs is NP-hard when $k \ge 6$. 
Thus, 
we require an assumption for our algorithm to work. 
One solution to this difficulty is to impose the connectivity of $F_U$, 
i.e.,\ 
when expanding $U \in \U$, 
we require that 
there always exists 
$F_U$ satisfying \eqref{EQexpandt} and that 
$(U, F_U)$ is connected. 
If $t = 2$, 
this property amounts to the Hamilton-laceability of $G[U]$ (see \cite{Tak17}). 

It is clear that $\ktt$-free $t$-matchings in bipartite graphs \cite{Fra03} 
also satisfy this assumption. 
Under this assumption, 
$F \subseteq E$ satisfying \eqref{EQdegconst} is feasible if and only if 
$F$ does not contain a $K_{t,t}$ of natural vertices as a subgraph. 

\subsubsection{Matchings and Even Factors in Nonbipartite Graphs}
\label{SECappef}

Since the nonbipartite matching problem is reduced to the even factor problem, 
it suffices to discuss only the even factor problem. 
Here, 
let $D=(V,A)$ be an \ocs{} digraph and define $\h{G}=(\hat{V},\hat{E})$ and $\U$ by \eqref{EQef}. 
Recall that, 
if $D$ is \ocs{}, 
then 
$\hat{G}[U]$ is a symmetric bipartite graph for each $U \in \U$. 
In this case, 
our algorithm is performed recursively, 
i.e.,\ 
a $1$-matching $F \subseteq E$ in a shrunk graph $G$ is feasible 
if $|F[U']| \le |U'|/2 - 1$ for $U' = \rmn{U} \cup \rmp{U}$ with $U \in \U$. 
This can be checked in $\order{n}$ time, 
and 
Procedure \textsc{Shrink}($U'$) is executed when a perfect matching in $G[U']$ is found in our solution. 
See Fig.\ \ref{FIGefShrink} for an illustration. 

\begin{figure}
\centering
\includegraphics[height=.25\textheight]{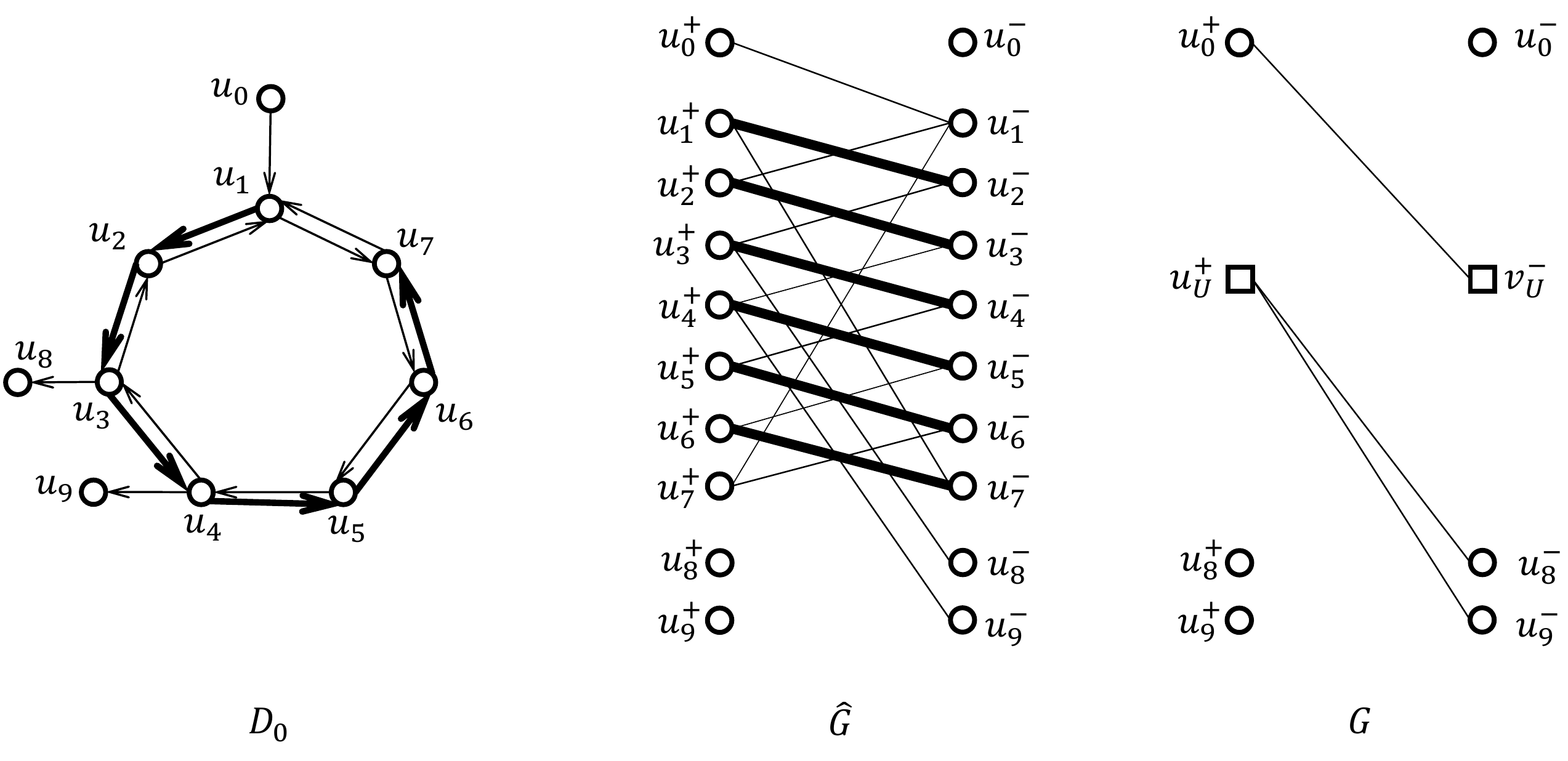}
\caption{The maximum even factor problem in $D_0$ is reduced to the \ufeas{} $1$-matching problem in $\hat{G}$. 
If we find arc $(u_7^+, u_1^-)$ as an $S$-$T$ path in the auxiliary digraph, 
we shrink $U=\{u_1^+,\ldots, u_7^+, u_1^-,\ldots, u_7^-\}$ to obtain $G$.}
\label{FIGefShrink}
\end{figure}

In \textsc{Expand}($G,F$), 
we repeat expanding a maximal shrunk vertex set $U$, 
where the proper shrunk subsets of $U$ remain shrunk. 
We repeat this step until the original graph $\hat{G}$ is reconstructed. 
See Fig.\ \ref{FIGefExpand1} for an illustration of expanding $U$. 
Without loss of generality, 
we can denote the perfect matching in $G[U']$ by $\bigcup_{i=1}^{k}\{\p{u_i},\m{u_{i+1}}\}$, 
where $k=2k'+1$ is odd and $u_{k+1} = u_1$. 
Furthermore, 
assume that $\m{u_1}$ and $\p{u_j}$ are incident to an edge in $F[U', V \sm U']$. 
Now, 
if $j = 2j'+1$ is odd, 
let $F_{U'} = \bigcup_{i=1}^{j-1}\{\p{u_{i}}, \m{u_{i+1}}\} \cup \bigcup_{i=j'+1}^{k'}\{ \{\p{u_{2i}}, \m{u_{2i+1}}\}, \{\p{u_{2i+1}}, \m{u_{2i}}\} \}$. 
If $j=2j'$  is even, 
then let 
$F_{U'} = \bigcup_{i=1}^{j'-1}\{ \{\p{u_{2i}}, \m{u_{2i+1}}\}, \{\p{u_{2i+1}}, \m{u_{2i}}\}\} \cup \bigcup_{i=j}^{k}\{\p{u_{i+1}}, \m{u_{i}} \}$. 
It is straightforward that \textsc{Expand}($G,F$) can be performed in $\order{n}$ time. 

Note that this procedure is possible because $k$ is odd and $G[U']$ is symmetric. 
We also remark that this procedure corresponds to expanding an odd cycle in an even factor algorithm \cite{Pap07}, 
and expanding an odd cycle in Edmonds' blossom algorithm \cite{Edm65}. 

\begin{figure}
\centering
\includegraphics[height=.25\textheight]{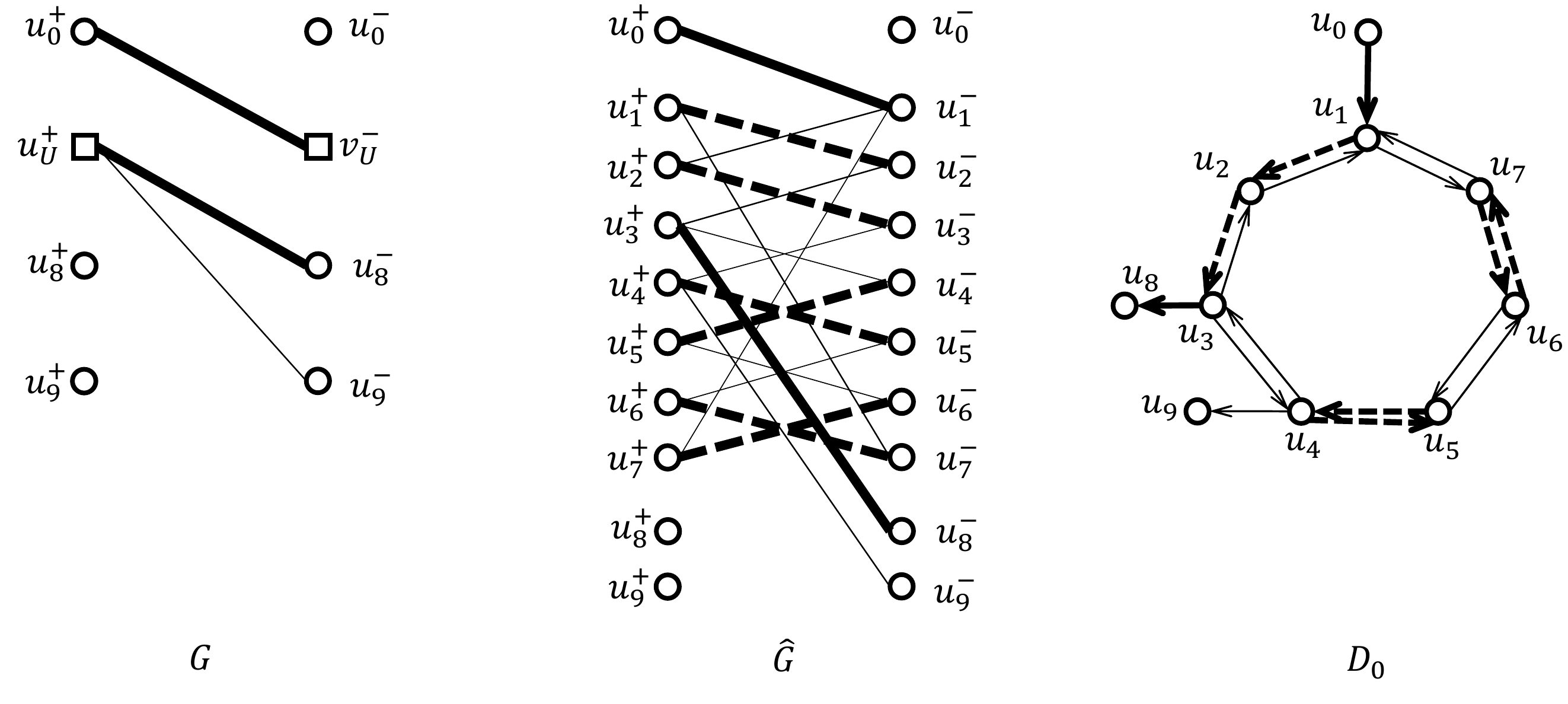}

\medskip

\includegraphics[height=.25\textheight]{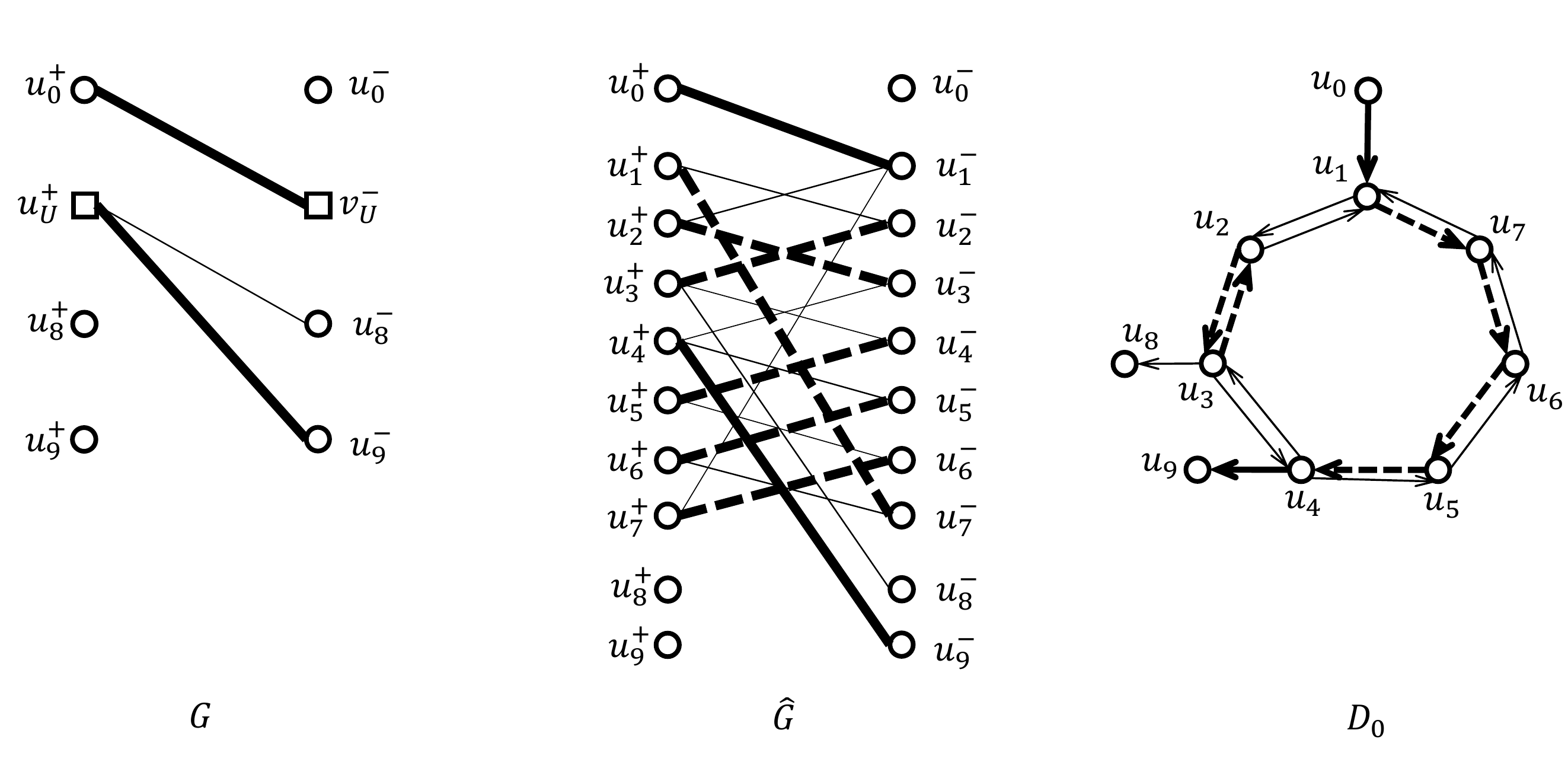}
\caption{Two types of expanding 
$U$.
The set of thick edges in $\hat{G}$ is our \ufeas{} $1$-matching, 
where the dashed edges are those added when expanding $U$. 
This \ufeas{} $1$-matching corresponds to the even factor of thick arcs in $D_0$.}
\label{FIGefExpand1}
\end{figure}

\subsubsection{Triangle-free 2-matchings}
Recall the instance of the \ufeas{} $t$-matching problem constructed in Sect.\ \ref{SECtri}. 
Here, we denote a graph obtained from $\hat{G}$ by repeated shrinkings by $G$. 
In $G$, 
a pair of vertices $u \in \p{V}$ and $v \in \m{V}$ is referred to as \emph{twins} if 
they are copies of the same original vertex or 
they are pseudovertices added by the same shrinking procedure. 

Here, 
a $1$-matching $F \subseteq E$ is infeasible only if 
it contains a matching of three edges covering three pairs of twins in the original graph $\hat{G}$. 
In other words, 
even if $F$ contains a matching of three edges covering three pairs of twins in $G$, 
it is feasible if the endvertices of those edges are not three pairs of twins in $\hat{G}$. 

For example, 
recall the instance in Fig.\ \ref{FIGtriangle}. 
In Fig.\ \ref{FIGtriAug}, 
$G$ is obtained from $\hat{G}$ by shrinking $U = \{u_3^+,u_4^+,u_5^+, u_3^-,u_4^-,u_5^-\}$, 
and we have a feasible edge set $\{\{ u_1^+, u_2^- \}, \{ u_2^+, v_U^- \}\}$. 
If we find an $S$-$T$ path $P_1$ consisting of a single arc resulting from $(u_5^+,u_1^-)$, 
then $F \sd P_1$ is feasible 
and $\mbox{\textsc{Augment}}(G,F,P)$ and $\mbox{\textsc{Expand}}(G,F)$ follow. 

In contrast, 
in Fig.\ \ref{FIGtriShrink}, 
suppose that we find an $S$-$T$ path $P_2$ consisting of a single arc $(u_3^+,u_1^-)$. 
Then, 
$F \sd P_2$ is \emph{not} feasible 
and $\mbox{\textsc{Shrink}}(W)$ follows, 
where 
$W = \{ u_1^+,u_2^+,u_3^+, u_1^-,u_2^-,u_3^- \}$. 
Note that 
the family of vertex sets shrunk by our algorithm corresponds to a \emph{triangle cluster} 
in the triangle-free $2$-matching algorithm due to Cornu\'ejols and Pulleyblank \cite{CP80}. 

\begin{figure}
\centering
\includegraphics[height=.18\textheight]{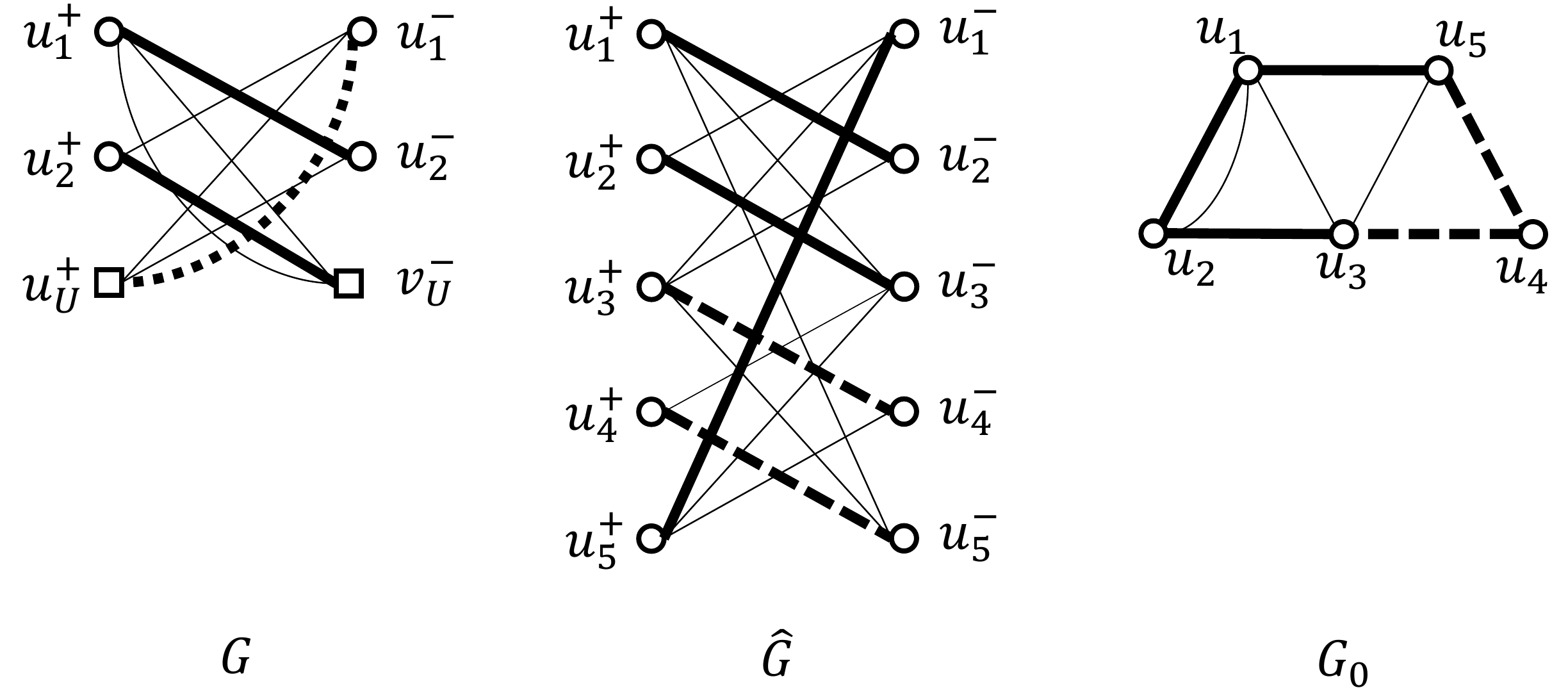}
\caption{If we find an $S$-$T$ path of a single arc resulting from $(u_5^+,u_1^-)$ (dotted edge in $G$), 
we execute $\mbox{\textsc{Augment}}(G,F,P)$ and $\mbox{\textsc{Expand}}(G,F)$. 
The set of thick edges in $\hat{G}$ is the obtained \ufeas{} $1$-matching, 
where the dashed edges are those added by $\mbox{\textsc{Expand}}(G,F)$. 
This \ufeas{} $1$-matching corresponds to a triangle-free $2$-matching 
(indicated by the thick edges) 
in $G_0$.}
\label{FIGtriAug}
\end{figure}

\begin{figure}
\centering
\includegraphics[width=.8\textwidth]{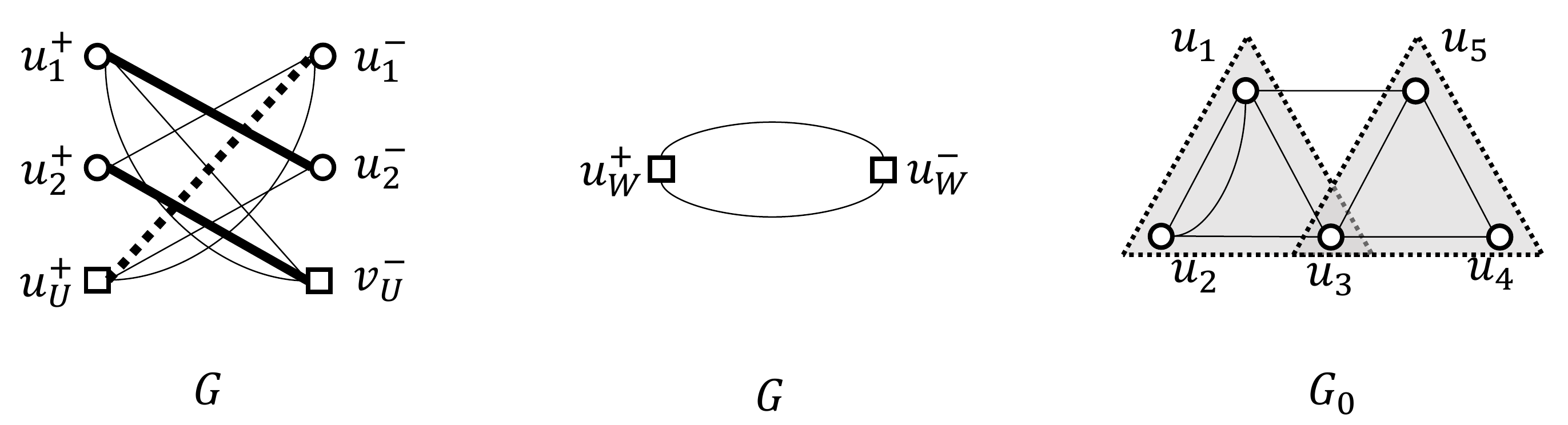}
\caption{If we find an $S$-$T$ path of a single arc resulting from $(u_3^+,u_1^-)$ (dotted edge in $G$), 
we execute $\mbox{\textsc{Shrink}}(W)$, 
where $W = \{ u_1^+,u_2^+,u_3^+, u_1^-,u_2^-,u_3^- \}$. 
The parallel edges between $\p{u_{W}}$ and $\m{u_{W}}$ result from $\{\p{u_1},\m{u_5}\}$ and $\{\p{u_5},\m{u_1}\}$. 
The shrunk triangles $\{u_1,u_2,u_3\}$ and $\{u_3,u_4,u_5\}$ in $G_0$ form a triangle cluster \cite{CP80}.}
\label{FIGtriShrink}
\end{figure}

For \textsc{Expand}($G,F$), 
we expand each shrunk $U \in \U$ one by one. 
We add two edges from $\hat{E}[U]$ to $F$ such that $F$ remains a $1$-matching, 
which always obtains the $\U$-feasibility of the output $F$. 

It is clear that the feasibility of $F$ can be determined in $\order{n}$ time, 
and 
that of $F_i$ in $\mbox{\textsc{FindViolatingSet}}(G,F,P)$ can be determined in a constant time for each $i=1,\ldots , l$. 
In addition,
\textsc{Expand}($G,F$) needs $\order{n}$ time.

\subsubsection{Matroids and Arborescences}

Recall the instances constructed in Sect.\ \ref{SECarb}. 
If Algorithm \ref{ALGunweighted} is applied to finding a maximum independent set of a matroid, 
then 
an $S$-$T$ path always consists of one arc, 
i.e.,\ 
the solution is greedily augmented. 
However, 
Procedure $\mbox{\textsc{Shrink}}(U)$ may occur, 
which is simply the contraction of a circuit for matroids. 

If Algorithm \ref{ALGunweighted} is applied to finding an arborescence, 
then 
an $S$-$T$ path always consists of one arc 
and the solution is greedily augmented. 
Here, 
Procedure $\mbox{\textsc{Shrink}}(U)$ corresponds to shrinking a directed cycle.

\subsection{$C_{4k+2}$-free $2$-matchings: New Example in Our Framework}

So far, 
we have viewed that some problems in the literature are described as the \ufeas{} $t$-matching problem in bipartite graphs under the assumption that $(G,\U,t)$ admits expansion. 
Here we exhibit a new problem which falls in this framework. 

Let $G=(V,E)$ be a simple bipartite graph. 
A $2$-matching $F \subseteq E$ is called \emph{$C_{4k+2}$-free} if $F$ excludes cycles of length $4k+2$ for every positive integer $k$. 
In other words, 
the length of a cycle in $F$ must be a multiple of four. 

Now $C_{4k+2}$-free $2$-matchings are described as \ufeas{} $2$-matchings, where 
$$\U = \{U \subseteq V \mid \mbox{$|\p{U}|=|\m{U}|=2k+1$ for some positive integer $k$}\}.$$
A solvable class of this form of the \ufeas{} $2$-matching problem is obtained by recalling the instances for even factors: 
a class where $G[U]$ is a symmetric bipartite graph 
and all of the twins in $G[U]$ are connected by an edge for every $U \in \U$. 
To be precise, 
every $U \in\U$ is described as 
$\p{U}=\{u_1,\ldots, u_{2k+1}\}$ and $\m{U}= \{v_1,\ldots, v_{2k+1}\}$, 
where 
$\{u_i, v_j\} \in E$ if and only if $\{u_j,v_i\} \in E$, 
and 
$\{u_i, v_i\} \in E$ for each $i=1,\ldots, 2k+1$. 
Then, 
it is straightforward to see that 
$(G,\U,t)$ admits expansion by following the arguments for even factors in Sect.\ \ref{SECbief} and \ref{SECappef}.

\section{Weighted \ufeas{} $t$-matching}
\label{SECw}

In this section, 
we extend the min-max theorem and the algorithm presented in Sect.\ \ref{SECunweighted} to the maximum-weight \ufeas{} $t$-matching problem. 
Recall that $G$ is a simple bipartite graph 
and $(G,\U,t)$ admits expansion. 
We further assume that 
$w$ is vertex-induced on each $U \in \U$, 
which commonly extends the assumptions for the maximum-weight square-free and even factor problems. 

\subsection{Linear Program}

Here, 
we describe a linear programming relaxation of the 
the maximum-weight \ufeas{} $t$-matching problem 
in variable $x \in \RR\sp{E}$: 
\begin{alignat}{3}
\mbox{(P)}\quad 	&{}\mbox{maximize} \quad 	{}&&{}	\sum_{e \in E}w(e) x(e) 							{}&& \\
					&{}\mbox{subject to} \quad 	{}&&{}	x(\delta(v)) \le t \quad							{}&&{}(v \in V), \\
					\label{EQufeas}
					&{}							{}&&{}	x(E[U]) \le \left\lfloor \frac{t|U|-1}{2} \right\rfloor \quad 	{}&&{}(U \in \U), \\
					&{}							{}&&{}	0 \le x(e) \le 1  \quad 							{}&&{}(e \in E). 
\end{alignat}
Note that 
Constraint \eqref{EQufeas}, which describes $\U$-feasibility, 
is a common extension of the blossom constraint for 
the nonbipartite matching problem ($t=1$), 
and 
the subtour elimination constraints for the TSP ($t=2$). 

Its dual program 
in variables 
$p\in \RR\sp{V}$,
$q\in \RR\sp{E}$, 
and 
$r \in \RR\sp{\U}$ 
is given as follows: 
\begin{alignat}{3}
\mbox{(D)}\quad 	&{}\mbox{minimize} \quad 	{}&&{}	t\sum_{v \in V}p(v) + \sum_{e \in E}q(e) + \sum_{U\in \U}\left\lfloor \frac{t|U|-1}{2} \right\rfloor r(U)  	{}&& \\
					&{}\mbox{subject to} \quad 	{}&&{}	p(u) + p(v) + q(e) + \sum_{U \in \U\colon e \in E[U]}r(U) \ge w(e)\quad 						{}&&{}(e=\{u,v\} \in E), \\
					&{}							{}&&{}	p(v) \ge 0 \quad 																				{}&&{}(v \in V), \\
					&{}							{}&&{}	q(e) \ge 0 \quad 																				{}&&{}(e \in E), \\
					&{}							{}&&{}	r(U) \ge 0 \quad 																				{}&&{}(U \in \U). 
\end{alignat}
We define $w'\in \RR\sp{E}$ by 
	\begin{align*}
	w'(e)
	= 
	p(u) + p(v)+ q(e) + \sum_{U \in \U\colon e \in E[U]}r(U) - w(e) \quad(e=\{u,v\} \in E). 
	\end{align*}
The complementary slackness conditions for (P) and (D) are as follows. 
\begin{alignat}{2}
\label{EQcsx}
&{}x(e) > 0 \Longrightarrow w'(e)=0 						\quad{}&&{} 	(e \in E), \\
\label{EQcsp}
&{}p(v) > 0 \Longrightarrow x(\delta (v)) = 2 				\quad{}&&{} 	(v \in V), \\
\label{EQcsq}
&{}q(e) > 0 \Longrightarrow x(e) = 1 						\quad{}&&{} 	(e \in E), \\
\label{EQcsr}
&{}r(U) > 0 \Longrightarrow x(E[U]) = \left\lfloor \frac{t|U|-1}{2} \right\rfloor 	\quad{}&&{} 	(U \in \U). 
\end{alignat}

\subsection{Primal-Dual Algorithm}

In this section, 
we demonstrate a combinatorial primal-dual algorithm for the maximum-weight \ufeas{} $t$-matching problem 
in bipartite graphs, 
where $(G,\U,t)$ admits expansion and 
$w$ is vertex-induced for each $U \in \U$. 

We maintain primal and dual feasible solutions 
that satisfy \eqref{EQcsx}, \eqref{EQcsq}, \eqref{EQcsr}, 
and \eqref{EQcsp} for $v \in \m{V}$. 
The algorithm terminates when \eqref{EQcsp} is obtained for every $v \in \p{V}$. 
Again, we denote the input graph by $\hat{G} = (\h{V}, \h{E})$, 
and 
the graph in hand, 
i.e., the graph resulting from possibly repeated shrinkings, 
by $G=(V,E)$. 
The variables in the algorithm are 
$F \subseteq E$, 
$p \in \RR\sp{\h{V}}$, 
$q \in \RR\sp{\h{E}}$, 
and 
$r \in \RR\sp{\U}$. 
Note that $p$ and $q$ are always defined on the original vertex and edge sets, 
respectively. 

Initially, 
we set
\begin{align}
&F = \emptyset, &
&p(v) = 
\begin{cases}
\max\{w(e) \mid e \in \delta(v)\} 	& (v \in \p{V}), \\
0 									&(v \in \m{V}), 
\end{cases}\notag\\
& q(e) = 0 \quad (e \in E), &
& r(U) = 0 \quad (U \in \U).
\label{EQinitial}
\end{align}
The auxiliary digraph $D$ is constructed as follows. 
Here, 
the major differences from Sect.\ \ref{SECalg} are that 
we only use an edge $e$ with $w'(e)=0$, 
and a vertex in $\p{V}$ can become a sink vertex. 

\paragraph{Procedure \textsc{ConstructAuxiliaryDigraph}$(G,F,p,q,r)$.}
Here, 
we define a digraph $(V,A)$ by
$$A = 	\{ (\partial^+e,\partial^-e) \mid \mbox{$e\in E \sm F$, $w'(e)=0$}   \} 
		\cup \{ (\partial^-e,\partial^+e) \mid e=\{u,v\} \in F   \}.$$
The sets of source vertices $S \subseteq \p{V}$ and 
sink vertices $T \subseteq \p{V} \cup \m{V}$ are
defined by 
\begin{alignat*}{2}
&S = {}&&{}	\{ u \in \p{\vn} \mid \mbox{$\deg_F(v) \le t-1$, $p(u) > 0$} \}, \\
		&&&{}	\cup \{ \p{u_U} \in \p{\vp} \mid \mbox{$\deg_F(\p{u_U}) = 0$, $p(u) > 0$ for some $u \in U$} \}\\
&T = {}&&{}	\{ v \in \m{\vn} \mid \deg_F(v) \le t-1 \} \cup \{ \m{v_U} \in \m{\vp} \mid \deg_F(\m{v_U}) = 0 \}\\
		&&&{}	\cup \{ u \in \p{\vn} \mid \mbox{$\deg_F(u) = t$, $p(u) = 0$} \} \\
		&&&{}	\cup \{ \p{u_U} \in \p{\vp} \mid \mbox{$\deg_F(u_U^+) = 1$, $p(u) = 0$ for some $u \in U$} \}.
\end{alignat*}
Return $D=(V,A;S,T)$,

\medskip 

Suppose that $D$ has a directed path $P$ from $S$ to $T$, 
and let $F' := F \sd P$. 
If $F'$ is feasible, 
we execute $\mbox{\textsc{Augment}}(G,F,P)$, 
which is the same as in Sect.\ \ref{SECalg}. 
Note that, 
if $P$ ends in a vertex in $T \cap \p{V}$, 
then $|F|$ does not increase. 
However, 
in this case, 
the number of vertices satisfying \eqref{EQcsp} increases by one, 
and we get closer to the termination condition (achieving \eqref{EQcsp} at every vertex). 
If $F'$ is not feasible, 
we apply \textsc{ViolatingSet}($G,F,P$) as in Sect.\ \ref{SECalg}. 
For the output $U$ of \textsc{ViolatingSet}($G,F,P$), 
if $p(u) = 0$ holds for some $u \in \p{U}$, 
then we execute 
$\mbox{\textsc{Modify}}(G,F,U)$ below.  
Otherwise, 
we apply 
$\mbox{\textsc{Shrink}}(U)$ as in Sect.\ \ref{SECalg}. 

\paragraph{Procedure $\mbox{\textsc{Modify}}(G,F,U)$.}
Let $u^* \in \p{U}$ satisfy $p(u^*) = 0$. 
Then find $K \subseteq E[U]$ such that 
\begin{align*}
\deg_{K}(u)= 
\begin{cases}
	t 				& (u \in \p{\uo} \sm \{u\sp{*}\}), \\
	t-1 			& (u = u^*), \\
	0 				& (u = \p{u_{U'}} \in \p{\up}, u^* \in U'), \\
	\deg_{F[U]}(u) 	& (u \in \m{\un} \cup \m{\up}). 
\end{cases}
\end{align*}
Here, return $F:= (F \sm F[U]) \cup K$. 

\medskip

If $D$ does not have a directed path from $S$ to $T$, 
then 
update the dual variables $p$, $q$, and $r$ 
by procedure $\mbox{\textsc{UpdateDualSolution}}(G,F,p,q,r)$ described below. 
\paragraph{Procedure $\mbox{\textsc{UpdateDualSolution}}(G,F,p,q,r)$.}
Let $R \subseteq V$ be the set of vertices reachable from $S$ in the auxiliary digraph $D$. 
Then,
\begin{align*}
&{}p(v) := 
		\begin{cases}
			p(v) - \epsilon 		& (v \in \p{\hat{R}}), \\
			p(v) + \epsilon 		& (v \in \m{\hat{R}}), \\
			p(v)					& (v \in \hat{V} \sm \hat{R}),
		\end{cases} \\
&{}q(e) := 
		\begin{cases}
			q(e) + \epsilon 	&(\mbox{$\partial^+e \in \p{\hat{R}}$, $\partial^-e \in \m{\hat{V}}\sm\m{\hat{R}}$}), \\
			q(e)				& (v \in \m{\hat{V}} \sm \m{\hat{R}}), 
		\end{cases} \\
&{}r(U) := 
		\begin{cases}
			r(U) + \epsilon 	& (\p{u_U} \in \p{R}, \m{v_U} \in \m{V} \sm \m{R}), \\
			r(U) - \epsilon 	& (\p{u_U} \in \p{V} \sm \p{R}, \m{v_U} \in \p{R}), \\
			r(U) 			 	& (\mbox{otherwise}), 
		\end{cases}
\end{align*}
where 
\begin{align*}
{}&{}\epsilon   = \min\{\epsilon_1,\epsilon_2,\epsilon_3\}, \quad
\epsilon_1 = \min\{w'(\{u,v\}) \mid u\in \p{\hat{R}}, v \in \m{\hat{V}} \sm \m{\hat{R}}\}, 
\\
{}&{}
\epsilon_2 = \min\{p(u) \mid u \in \hat{\p{R}}\}, 
\quad
\epsilon_3 = \min\{r(U) \mid \p{u_U} \in \p{\hat{V}} \sm \p{\hat{R}}, \m{v_U} \in \m{\hat{R}}   \}. 
\end{align*}
Then 
return $(p,q,r)$. 

\medskip

Finally, 
we expand every $U$ satisfying $r(U)=0$ after 
\textsc{Augment}$(G,F,P)$, 
\textsc{Modify}$(G,F,U)$, 
and 
\textsc{UpdateDualSolution}$(G,F,p,q,r)$. 
If any $U' \subsetneq U$ satisfies $r_{U'}>0$, 
which implies that $U'$ had been shrunk before $U$ was shrunk, 
then $U'$ is maintained as shrunk.  
\paragraph{Procedure $\mbox{\textsc{Expand}}(G,F,r)$.}
For each shrunk $U \in \U$ with $r(U)=0$, 
execute the following procedures. 
Update $G$ by replacing $\p{u_{U}}$ and $\m{v_U}$ 
by the graph induced by $\un \cup \up$ just before \textsc{Shrink}($U$) is applied. 
Determine $F_U \subseteq E[\un \cup \up]$ 
of $\lfloor (t|\un| + |\up|-1)/2\rfloor - 1$ edges 
such that $F' = F \cup F_U$ can be extended to a \ufeas{} $t$-matching in $\hat{G}$. 
Then return $F:= F'$. 

\medskip

The pseudocode of the maximum-weight \ufeas{} $t$-matching algorithm is presented in Algorithm \ref{ALGweighted}. 
For complexity, 
it follows that one \textsc{DualUpdate}($G,F,p,q,r$) requires $\order{m}$ time, 
and it is executed $\order{n^3}$ times. 
This is the difference from the unweighted version; 
thus, 
the total complexity is $\order{t(n^3 (m+\alpha) + n^2 \beta)}$.

\begin{algorithm}[t]
\caption{Maximum-weight \ufeas{} $t$-matching}
\label{ALGweighted}
\label{alg:general}
\begin{algorithmic}[1]
\State Set $F,p,q,r$ by \eqref{EQinitial}
\While{Condition \eqref{EQcsp} is violated}
	\State $D \leftarrow \mbox{\textsc{AuxiliaryDigraph}}(G,x,p,q,r)$ 
	\If{$D$ has an $S$-$T$ path $P$}
		\If{$F \sd P$ is \ufeas}
			\State $F \leftarrow \mbox{\textsc{Augment}}(G,F,P)$ 
			\State $(G,F) \leftarrow \mbox{\textsc{Expand}}(G,F,r)$
		\Else
			\State $(F,U) \leftarrow \mbox{\textsc{ViolatingSet}}(G,F,P)$
			\If{$p(u)=0$ for some $u \in U$}
				\State $F \leftarrow \mbox{\textsc{Modify}}(G,F,U)$
				\State $(G,F) \leftarrow \mbox{\textsc{Expand}}(G,F,r)$
			\Else 
				\State $(G,F) \leftarrow \mbox{\textsc{Shrink}}(U)$
			\EndIf 
		\EndIf
	\Else 
		\State $(p,q,r) \leftarrow \mbox{\textsc{DualUpdate}}(G,F,p,q,r)$ 
		\State $(G,F) \leftarrow \mbox{\textsc{Expand}}(G,F,r)$
	\EndIf
\EndWhile
\State $(G,F) \leftarrow \mbox{\textsc{Expand}}(G,F)$ \\
\Return $(F,p,q,r)$
\end{algorithmic}
\end{algorithm}

It is clear that the optimal dual solution $(p,q,r)$ found by Algorithm \ref{ALGweighted} is integer if 
the edge weight $w$ is integer. 
Thus, 
Algorithm \ref{ALGweighted} constructively proves the following theorem for the integrality of (P) and (D). 
This is a common extension of dual integrality theorems for 
nonbipartite matchings \cite{CM78}, 
even factors \cite{KM04}, 
triangle-free $2$-matchings \cite{CP80}, 
square-free $2$-matchings \cite{Mak07}, 
and 
branchings \cite{Edm67}. 

\begin{theorem}
If 
$(G,\U,t)$ admits expansion and 
$w$ is vertex-induced 
on each $U \in \U$, 
then 
the linear program \textsc{(P)} has an integer optimal solution. 
Moreover, 
the linear program
\textsc{(D)} also an integer optimal solution such that 
the number of sets $U \in \U$ with $r(U)>0$ is at most $n/2$. 
\hfill \qedsymbol
\end{theorem}

\section{Conclusion}
\label{SECconcl}

We have presented a new framework for the optimal \ufeas{} $t$-matching problem 
and established a min-max theorem and combinatorial algorithm under the reasonable assumption that 
$G$ is bipartite, 
$(G,\U,t)$ admits expansion, 
and 
$w$ is vertex-induced on each $U \in \U$. 
Under this assumption, 
our problem can describe a number of generalization of the matching problem, 
such as the matching and triangle-free $2$-matching problems in nonbipartite graphs, 
the square-free $2$-matching problem in bipartite graphs, 
and matroids and arborescences. 
We have also obtained a new class of the restricted $2$-matching problem, the $C_{4k+2}$-free $2$-matching problem, 
which can be solved efficiently under a corresponding assumption. 

It is noteworthy that the $\U$-feasibility is a common generalization of 
the blossom constraints for the nonbipartite matching problem 
and 
the subtour elimination constraints for the TSP\@. 
We expect that this unified perspective will provide a 
new approach to the TSP utilizing matching and matroid theories. 

\section*{Acknowledgements}

I thank Yutaro Yamaguchi for the helpful comments regarding the draft of the paper. 
I am also thankful to the anonymous referees for their careful reading and comments. 
This work has been supported by 
JSPS KAKENHI Grant Numbers 16K16012, 
25280004, 
and 26280001, 
and 
CREST, JST, Grant Number JPMJCR1402, Japan.

\bibliographystyle{myjorsj2}
\bibliography{../../../../../refs}

\begin{thebibliography}{10}

\bibitem{Bab12}
M.A. Babenko: Improved algorithms for even factors and square-free simple
  $b$-matchings, {\em Algorithmica}, 64 (2012), 362--383.

\bibitem{BK12}
K.~B\'{e}rczi and Y.~Kobayashi: An algorithm for $(n-3)$-connectivity
  augmentation problem: Jump system approach, {\em Journal of Combinatorial
  Theory, Series~B}, 102 (2012), 565--587.

\bibitem{BV10}
K.~B\'{e}rczi and L.A. V\'{e}gh: Restricted $b$-matchings in degree-bounded
  graphs, in F.~Eisenbrand and B.~Shepherd, eds., {\em Integer Programming and
  Combinatorial Optimization: Proceedings of the 14th IPCO, LNCS~6080},
  Springer, 2010,  43--56.

\bibitem{BIT13}
S.~Boyd, S.~Iwata and K.~Takazawa: Finding 2-factors closer to {TSP} tours in
  cubic graphs, {\em SIAM Journal on Discrete Mathematics}, 27 (2013),
  918--939.

\bibitem{CL65}
Y.J. Chu and T.H. Liu: On the shortest arborescence of a directed graph, {\em
  Scientia Sinica}, 14 (1965), 1396--1400.

\bibitem{CP80}
G.~Cornu\'{e}jols and W.~Pulleyblank: A matching problem with side conditions,
  {\em Discrete Mathematics}, 29 (1980), 135--159.

\bibitem{CG97}
W.H. Cunningham and J.F. Geelen: The optimal path-matching problem, {\em
  Combinatorica}, 17 (1997), 315--337.

\bibitem{CG01}
W.H. Cunningham and J.F. Geelen: Vertex-disjoint dipaths and even dicircuits,
  unpublished, 2001.

\bibitem{CM78}
W.H. Cunningham and A.B. {Marsh, III}: A primal algorithm for optimum matching,
  {\em Mathematical Programming Study}, 8 (1978), 50--72.

\bibitem{Edm65}
J.~Edmonds: Paths, trees, and flowers, {\em Canadian Journal of Mathematics},
  17 (1965), 449--467.

\bibitem{Edm67}
J.~Edmonds: Optimum branchings, {\em Journal of Research National Bureau of
  Standards{\em ,} Section~B}, 71 (1967), 233--240.

\bibitem{Fra03}
A.~Frank: Restricted $t$-matchings in bipartite graphs, {\em Discrete Applied
  Mathematics}, 131 (2003), 337--346.

\bibitem{FS02}
A.~Frank and L.~Szeg\H{o}: Note on the path-matching formula, {\em Journal of
  Graph Theory}, 41 (2002), 110--119.

\bibitem{HartD}
D.~Hartvigsen: {\em Extensions of Matching Theory}, Ph.D. thesis, Carnegie
  Mellon University, 1984.

\bibitem{Hart06}
D.~Hartvigsen: Finding maximum square-free 2-matchings in bipartite graphs,
  {\em Journal of Combinatorial Theory{\em ,} Series~B}, 96 (2006), 693--705.

\bibitem{HL11}
D.~Hartvigsen and Y.~Li: Maximum cardinality simple $2$-matchings in subcubic
  graphs, {\em SIAM Journal on Optimization}, 21 (2011), 1027--1045.

\bibitem{HL13}
D.~Hartvigsen and Y.~Li: Polyhedron of triangle-free simple $2$-matchings in
  subcubic graphs, {\em Mathematical Programming}, 138 (2013), 43--82.

\bibitem{Iri79}
M.~Iri: A review of recent work in Japan on principal partitions of matroids
  and their applications, {\em Annals of the New York Academy of Sciences}, 319
  (1979), 306--319.

\bibitem{IF81}
M.~Iri and S.~Fujishige: Use of matroid theory in operations research, circuits
  and systems theory, {\em International Journal of Systems Science}, 12
  (1981), 27--54.

\bibitem{IT08}
S.~Iwata and K.~Takazawa: The independent even factor problem, {\em SIAM
  Journal on Discrete Mathematics}, 22 (2008), 1411--1427.

\bibitem{KS08}
T.~Kaiser and R.~\v{S}krekovski: Cycles intersecting edge-cuts of prescribed
  sizes, {\em SIAM Journal on Discrete Mathematics}, 22 (2008), 861--874.

\bibitem{KM04}
T.~Kir\'{a}ly and M.~Makai: On polyhedra related to even factors, in
  D.~Bienstock and G.L. Nemhauser, eds., {\em Integer Programming and
  Combinatorial Optimization: Proceedings of the 10th IPCO, LNCS~3064},
  Springer, 2004,  416--430.

\bibitem{Kir99}
Z.~Kir\'{a}ly: {$C_4$}-free $2$-factors in bipartite graphs, Technical Report
  TR-2001-13, Egerv\'{a}ry Research Group, 1999, {\tt www.cs.elte.hu/egres}.

\bibitem{Kir09}
Z.~Kir{\'a}ly: Restricted $t$-matchings in bipartite graphs, Technical Report
  QP-2009-04, Egerv{\'a}ry Research Group, 2009, {\tt www.cs.elte.hu/egres}.

\bibitem{Kob14}
Y.~Kobayashi: Triangle-free $2$-matchings and {M}-concave functions on jump
  systems, {\em Discrete Applied Mathematics}, 175 (2014), 35--42.

\bibitem{KST12}
Y.~Kobayashi, J.~Szab\'{o} and K.~Takazawa: A proof of {Cunningham's}
  conjecture on restricted subgraphs and jump systems, {\em Journal of
  Combinatorial Theory, Series~B}, 102 (2012), 948--966.

\bibitem{KT09}
Y.~Kobayashi and K.~Takazawa: Even factors, jump systems, and discrete
  convexity, {\em Journal of Combinatorial Theory, Series~B}, 99 (2009),
  139--161.

\bibitem{KY12}
Y.~Kobayashi and X.~Yin: An algorithm for finding a maximum $t$-matching
  excluding complete partite subgraphs, {\em Discrete Optimization}, 9 (2012),
  98--108.

\bibitem{LP09}
L.~Lov\'{a}sz and M.D. Plummer: {\em Matching Theory}, AMS Chelsea Publishing,
  Providence, 2009.

\bibitem{Mak07}
M.~Makai: On maximum cost {$K_{t,t}$}-free $t$-matchings of bipartite graphs,
  {\em SIAM Journal on Discrete Mathematics}, 21 (2007), 349--360.

\bibitem{Mur06}
K.~Murota: M-convex functions on jump systems: A general framework for
  minsquare graph factor problem, {\em SIAM Journal on Discrete Mathematics},
  20 (2006), 231--226.

\bibitem{Pap04}
G.~Pap: A {TDI} description of restricted $2$-matching polytopes, in
  D.~Bienstock and G.L. Nemhauser, eds., {\em Integer Programming and
  Combinatorial Optimization: Proceedings of the 10th IPCO, LNCS~3064},
  Springer, 2004,  139--151.

\bibitem{Pap07}
G.~Pap: Combinatorial algorithms for matchings, even factors and square-free
  2-factors, {\em Mathematical Programming}, 110 (2007), 57--69.

\bibitem{PS04}
G.~Pap and L.~Szeg\H{o}: On the maximum even factor in weakly symmetric graphs,
  {\em Journal of Combinatorial Theory{\em ,} Series~B}, 91 (2004), 201--213.

\bibitem{Sim78}
G.J. Simmons: Almost all $n$-dimensional rectangular lattices are {Hamilton}
  laceable, in {\em Proceedings of the 9th Southeastern Conference on
  Combinatorics, Graph Theory, and Computing}, Congressus Numerantium 21, 1978,
   649--661.

\bibitem{SS04}
B.~Spille and L.~Szeg\H{o}: A Gallai-Edmonds-type structure theorem for
  path-matchings, {\em Journal of Graph Theory}, 46 (2004), 93--102.

\bibitem{SW02}
B.~Spille and R.~Weismantel: A generalization of Edmonds' matching and matroid
  intersection algorithms, in W.J. Cook and A.S. Schulz, eds., {\em Integer
  Programming and Combinatorial Optimization: Proceedings of the 9th IPCO,
  LNCS~2337}, Springer, 2002,  9--20.

\bibitem{Tak08}
K.~Takazawa: A weighted even factor algorithm, {\em Mathematical Programming},
  115 (2008), 223--237.

\bibitem{Tak09}
K.~Takazawa: A weighted {$K_{t,t}$}-free $t$-factor algorithm for bipartite
  graphs, {\em Mathematics of Operations Research}, 34 (2009), 351--362.

\bibitem{Tak10}
K.~Takazawa: Even factors: Algorithms and structure, in S.~Iwata, ed., {\em
  RIMS K\^oky\^uroku Bessatsu, B23}, Research Institute for Mathematical
  Sciences, Kyoto University, 2010,  233--252.

\bibitem{Tak12wief}
K.~Takazawa: A weighted independent even factor algorithms, {\em Mathematical
  Programming}, 132 (2012), 261--276.

\bibitem{Tak17ipco}
K.~Takazawa: Excluded $t$-factors: A unified framework for nonbipartite
  matchings and restricted 2-matchings, in F.~Eisenbrand and J.~Koenemann,
  eds., {\em Integer Programming and Combinatorial Optimization: Proceedings of
  the 19th IPCO, LNCS 10328}, Springer, 2017,  430--441.

\bibitem{Tak17DAM}
K.~Takazawa: Decomposition theorems for square-free 2-matchings in bipartite
  graphs, {\em Discrete Applied Mathematics}, to appear.

\bibitem{Tak17}
K.~Takazawa: Finding a maximum $2$-matching excluding prescribed cycles in
  bipartite graphs, {\em Discrete Optimization}, to appear.

\end{thebibliography}

\end{document}